\documentclass[12pt,a4paper]{article}

\usepackage{latexsym}
\usepackage{amsmath}
\usepackage{amssymb}
\usepackage{amsthm}
\usepackage{amscd}
\usepackage{mathrsfs}
\usepackage[all]{xy}
\usepackage{graphicx}
\usepackage{comment}
\usepackage{arydshln}

\newtheorem{thm}{Theorem}[section]
\newtheorem{prop}[thm]{Proposition}

\newtheorem{lem}[thm]{Lemma}
\newtheorem{cor}[thm]{Corollary}

\newtheorem{rem}[thm]{Remark}

\newtheorem*{thmn}{Theorem}

\makeatletter

\@addtoreset{equation}{section}
\makeatother

\makeatletter
\newcommand{\subsubsubsection}{\@startsection{paragraph}{4}{\z@}%
 {1.0\Cvs \@plus.5\Cdp \@minus.2\Cdp}%
 {.1\Cvs \@plus.3\Cdp}%
 {\reset@font\sffamily\normalsize}
 }
\makeatother
\DeclareMathOperator{\Spec}{Spec}

\DeclareMathOperator{\Gal}{Gal}

\DeclareMathOperator{\Hom}{Hom}

\DeclareMathOperator{\Ind}{Ind} 
\DeclareMathOperator{\Tr}{Tr}

\DeclareMathOperator{\Irr}{Irr}

\DeclareMathOperator{\Proj}{Proj}

\DeclareMathOperator{\SL}{SL}
\DeclareMathOperator{\GL}{GL}

\DeclareMathOperator{\Sp}{Sp}

\DeclareMathOperator{\SO}{SO}
\DeclareMathOperator{\SU}{SU}
\DeclareMathOperator{\PU}{PU}

\DeclareMathOperator{\oU}{U}

\DeclareMathOperator{\oO}{O}

\newcommand{\bA}{\mathbb{A}}

\newcommand{\bF}{\mathbb{F}}
\newcommand{\bG}{\mathbb{G}}

\newcommand{\bQ}{\mathbb{Q}}

\newcommand{\bZ}{\mathbb{Z}}

\newcommand{\cE}{\mathcal{E}}

\newcommand{\cO}{\mathcal{O}}

\newcommand{\fm}{\mathfrak{m}}

\newcommand{\sK}{\mathscr{K}}

\newcommand{\ol}{\overline}

\newcommand{\wt}{\widetilde}

\newcommand{\cf}{\textit{cf.\ }}

\begin{document}

\title
{Mod $\ell$ Weil representations and Deligne--Lusztig inductions for unitary groups}
\author{Naoki Imai and Takahiro Tsushima}
\date{}
\maketitle
\begin{abstract}
We study the mod $\ell$ Weil representation 
of a finite unitary group and related Deligne--Lusztig inductions. 
In particular, we study their decomposition as representations of a symplectic group, 
and give a construction of a mod $\ell$ Howe correspondence for 
$(\mathrm{Sp}_{2n},\mathrm{O}_2^-)$ including the case where $p=2$. 
\end{abstract}


\section{Introduction}
Let $q$ be a power of a prime number $p$. 
Weil representations of symplectic groups over 
$\mathbb{F}_q$ are studied in \cite{SaiRepsym} and \cite{HoweCharW} 
after \cite{Weiopuni} if $q$ is odd. 
Weil representations of general linear groups and 
unitary groups over $\mathbb{F}_q$ 
are constructed in \cite{GerWeil} for any $q$. 
The Howe correspondence is constructed using 
the Weil representations. 

In \cite{ITGeomHW}, we construct Weil representations of unitary groups by using cohomology of varieties over finite fields. 
More concretely, we consider the affine smooth variety
$X_n$ 
defined by $z^q+z=\sum_{i=1}^n x_i^{q+1}$
in $\mathbb{A}_{\mathbb{F}_{q^2}}^{n+1}$, where $n \geq 2$.  
Let $\mathbb{F}_{q,+}=\{x \in \mathbb{F}_{q^2} \mid x^q+x=0\}$. 
This variety admits an action of a finite unitary group 
$\oU_n(\bF_q)$ and a natural action of $\mathbb{F}_{q,+}$. 
Let $\ell \neq p$ be a prime number and 
$\psi \in \Hom (\mathbb{F}_{q,+},\overline{\mathbb{Q}}_{\ell}^{\times}) \setminus \{1\}$. 
Then the $\psi$-isotypic part 
$V_n =H_{\mathrm{c}}^n(X_{n,\overline{\mathbb{F}}_q},\overline{\mathbb{Q}}_{\ell})[\psi]$ realizes the Weil representation of $\oU_n(\bF_q)$ with a natural action of $\Gal (\ol{\bF}_{q}/\bF_{q^2})$. 
We can use this Galois action to construct Shintani lifts for Weil representations as in \cite{ITShinlift}. 
Further $V_n$ is isomorphic to middle  
 cohomology of a $\mu_{q+1}$-torsor over the complement of a Fermat hypersurface in a projective space as $\oU_n(\bF_q)$-representations.  
Let $S_n$ be the Fermat hypersurface 
defined by the homogeneous polynomial 
$\sum_{i=1}^n x_i^{q+1}=0$ in $\mathbb{P}_{\mathbb{F}_{q^2}}^{n-1}$. 
Let $Y_n=\mathbb{P}_{\mathbb{F}_{q^2}}^{n-1} \setminus S_n$. 
Here let $\oU_n(\bF_q)$ act on $\mathbb{P}_{\mathbb{F}_{q^2}}^{n-1}$ by multiplication.
Let $\widetilde{Y}_n$ be the affine smooth variety 
defined by $\sum_{i=1}^n x_i^{q+1}=1$ in $\mathbb{A}_{\mathbb{F}_{q^2}}^n$. 
The natural morphism 
$f \colon \widetilde{Y}_n \to Y_n;\ 
(x_1,\ldots,x_n) \mapsto [x_1:\cdots:x_n]$ is a $\mu_{q+1}$-torsor. These varieties appear as Deligne--Lusztig varieties. 

Let $\Lambda \in \{\ol{\bQ}_{\ell},\ol{\bF}_{\ell} \}$. 
Let $\mathscr{K}_{\chi}$ denote the 
sheaf of $\Lambda$-modules 
on $Y_n$ associated to a character 
$\chi^{-1} \colon \mu_{q+1} \to \Lambda^{\times}$
and the $\mu_{q+1}$-torsor $f$. 
We identify $\mu_{q+1}$ with the center of $\oU_n(\bF_q)$. 
For $\chi \in \Hom (\mu_{q+1}, \ol{\bQ}_{\ell}^{\times})$, 
we have an isomorphism 
$V_n[\chi] \simeq H_{\mathrm{c}}^{n-1}(
 Y_{n,\overline{\mathbb{F}}_q},\mathscr{K}_{\chi})$ 
as $\oU_n(\bF_q)$-representations (\cf \eqref{eq:VnA}), which we can write using Deligne--Lusztig induction (\cite[Proposition 6.1]{ITGeomHW}). 
In this paper, we study a modular coefficients case of this cohomology. 
Namely, we analyze 
\[
 H_{\mathrm{c}}^{n-1}(Y_{n,\overline{\mathbb{F}}_q},\mathscr{K}_{\xi}) 
\] 
as $\oU_n(\bF_q)$-representations over $\ol{\bF}_{\ell}$ 
for $\xi \in \Hom (\mu_{q+1},\ol{\bF}_{\ell}^{\times})$. 
We show that these representations are irreducible in the most cases, 
but it can be an extension of the trivial representation by an irreducible representations when $\ell \mid q+1$. This is contrary to the $\ol{\bQ}_{\ell}$-coefficients case, where they are all irreducible (\cf Proposition \ref{prop:Vn}). 
See Proposition \ref{ccc} for a more precise result. 

In \cite{ITGeomHW}, we consider a rational 
form of $X_{2n}$ over $\bF_q$, which is denoted by $X'_{2n}$.  
Using the Frobenius action coming from 
the rationality of $X'_{2n}$ over $\mathbb{F}_q$, we can obtain a representation of $\Sp_{2n}(\bF_q) \times \oO_2^-(\bF_q)$ on $H_{\mathrm{c}}^{2n}(X'_{2n,\overline{\mathbb{F}}_q},\overline{\mathbb{Q}}_{\ell})[\psi]$ (\cf \cite[\S 7.1.1]{ITGeomHW}), for 
which we simply write $W_{n,\psi}$. 
We also consider a rational form of $Y_{2n}$, which we denote by 
$Y'_{2n}$. 
For $\chi \in \Hom ( \mu_{q+1}, \Lambda^{\times})$, 
we can define a sheaf 
of $\Lambda$-modules $\mathscr{K}_{\chi}$ on $Y'_{2n}$ similarly, 
and then the cohomology 
$H_{\mathrm{c}}^{2n-1}(Y'_{2n,\overline{\mathbb{F}}_q},\mathscr{K}_{\chi})$ 
is regarded as an 
$\Sp_{2n}(\bF_q)$-representation. 
Similarly as above, 
we have an isomorphism  $W_{n,\psi}[\chi] \simeq 
H_{\mathrm{c}}^{2n-1}(Y'_{2n,\overline{\mathbb{F}}_q},\mathscr{K}_{\chi})$ 
as $\Sp_{2n}(\bF_q)$-representations 
for $\chi \in \Hom ( \mu_{q+1}, \ol{\bQ}_{\ell}^{\times})$. 
If $\chi^2=1$, we can define the plus part $H_{\mathrm{c}}^{2n-1}(Y'_{2n,\overline{\mathbb{F}}_q},\mathscr{K}_{\chi})^{+}$ and the minus part $H_{\mathrm{c}}^{2n-1}(Y'_{2n,\overline{\mathbb{F}}_q},\mathscr{K}_{\chi})^{-}$ as $\Sp_{2n}(\bF_q)$-representations using the Frobenius action. 
Any irreducible representation $\sigma$ of $\oO_2^-(\bF_q)$ over $\ol{\bQ}_{\ell}$ is attached to 
\[
 [\xi] \in \{ \xi \in \Hom (\mu_{q+1},\ol{\bF}_{\ell}^{\times}) \mid 
 \xi^2 \neq 1 \}/(\mathbb{Z}/2\mathbb{Z}) \quad \textrm{or} \quad  
 (\xi,\kappa) \in \Hom (\mu_{q+1},\mu_2 (\ol{\bF}_{\ell})) \times \{ \pm \} 
\]
as \cite[\S 7.2]{ITGeomHW}, where 
$1 \in \bZ/2\bZ$ acts on $\Hom (\mu_{q+1},\ol{\bF}_{\ell}^{\times})$ 
by $\xi \mapsto \xi^{-1}$. 
Then $W_{n,\psi}[\sigma]$ is isomorphic to 
\[
 H_{\mathrm{c}}^{2n-1}(Y'_{2n,\overline{\mathbb{F}}_q},\mathscr{K}_{\chi}) \quad \textrm{or} \quad 
 H_{\mathrm{c}}^{2n-1}(Y'_{2n,\overline{\mathbb{F}}_q},\mathscr{K}_{\chi})^{\kappa}
\]
accordingly. This realizes the Howe correspondence for 
the dual pair $\Sp_{2n}(\bF_q) \times \oO_2^-(\bF_q)$. 

The aim of this paper is to propose the modular coefficients version of this correspondence as a 
mod $\ell$ Howe correspondence for $\Sp_{2n}(\bF_q) \times \oO_2^-(\bF_q)$ (\cf \S \ref{sec:Formulation}) and study the mod $\ell$ correspondence. 
Our main result is the following: 


\begin{thmn}[Theorem \ref{thm}]
Assume that $\ell \neq 2$. The $\Sp_{2n}(\bF_q)$-representations 
\begin{align*}
 &H_{\mathrm{c}}^{2n-1}(Y'_{2n,\overline{\mathbb{F}}_q},\mathscr{K}_{\xi}) 
 \quad \textrm{for $[\xi] \in 
 \{ \xi \in \Hom (\mu_{q+1},\ol{\bF}_{\ell}^{\times}) \mid 
 \xi^2 \neq 1 \}/(\mathbb{Z}/2\mathbb{Z})$}, \\ 
 &H_{\mathrm{c}}^{2n-1}
 (Y'_{2n,\overline{\mathbb{F}}_q},\mathscr{K}_{\xi})^{\kappa} 
 \quad \textrm{for $\xi \in \Hom (\mu_{q+1},\mu_2 (\ol{\bF}_{\ell}))$, 
 $\kappa \in \{ \pm \}$} 
\end{align*}
are irreducible 
except the case where $\ell \mid q+1$ and 
$(\xi,\kappa)=(1,+)$, 
in which case 
$H_{\mathrm{c}}^{2n-1}
(Y'_{2n,\overline{\mathbb{F}}_q},\mathscr{K}_{\xi})^{\kappa}$ is 
a non-trivial extension of 
the trivial representation 
by an irreducible representation. 
Furthermore, 
the above representations 
have no irreducible constituent in common. 
\end{thmn}



In the following, we briefly introduce a content of each 
section. 
In \S \ref{Pre1}, we recall several fundamental facts 
proved in \cite{ITGeomHW}. 
In \S \ref{Pre3}, we recall general facts on \'{e}tale cohomology. 
In \S \ref{ssec:Lusbl}, we recall results on 
Lusztig series and $\ell$-blocks. 

In \S \ref{mod1}, we investigate the mod $\ell$ cohomology of $Y_n$ 
as a $\oU_n(\bF_q)$-representation. 
Our fundamental result is Proposition \ref{lp}. 
To show this proposition, we need a transcendental 
result in \cite{DimSinghyp} (\cf the proof of Lemma \ref{surj}). 
In \S \ref{mod2}, we prepare some geometric results on 
cohomology of $Y'_{2n}$. 
In \S \ref{sec:RepSP}, we study 
mod $\ell$ cohomology of $Y'_{2n}$ 
as an $\Sp_{2n}(\bF_q)$-representation for $\ell \neq 2$. 
In a modular 
representation theoretic view point,  
Brauer characters associated to $V_n$ and $W_{n,\psi}$
have been studied in \cite{GMSTCross}, 
\cite{GuTiCross} and \cite{HiMaLowuni} etc. 
Using these results, we study the cohomology 
of $Y_n$ and $Y'_{2n}$ as mentioned above. 


In \S \ref{mod3}, 
we formulate a mod $\ell$ Howe correspondence 
for $(\Sp_{2n},\oO_2^-)$ using mod $\ell$ cohomology of $Y'_{2n}$, 
and state our result in terms of the mod $\ell$ Howe correspondence. 

\subsection*{Acknowledgements}
The authors would like to thank the referee for giving helpful comments. 
This work was supported by JSPS KAKENHI Grant Numbers 20K03529, 21H00973, 22H00093.  

\subsection*{Notation}
Let $\ell$ be a prime number. 
For a finite abelian group $A$, let $A^{\vee}$
denote the character group $\Hom (A,\overline{\mathbb{Q}}_{\ell}^{\times})$. For a finite group $G$ and a finite-dimensional 
representation $\pi$ and an irreducible representation 
$\rho$ of $G$ over $\overline{\mathbb{Q}}_{\ell}$, 
let $\pi[\rho]$ denote the $\rho$-isotypic part of $\pi$. 
For the trivial representation $1$ of $G$, 
we often write $\pi^G$ for $\pi[1]$.  

Every scheme is equipped with 
the reduced scheme structure. 
For an integer $i \geq 0$, we write $\mathbb{A}^i$ and $\mathbb{P}^i$ 
for the $i$-dimensional 
affine space over $\overline{\mathbb{F}}_q$ and 
the $i$-dimensional 
projective space over $\overline{\mathbb{F}}_q$, respectively. 
We set $\bG_{\mathrm{m}}=\mathbb{A}^1 \setminus \{0\}$. 
For a scheme $X$ over a field $k$ and a field extension $l/k$, let $X_{l}$ denote the base change of $X$ to $l$.

\section{Preliminaries}\label{Pre}
\subsection{Weil representation of unitary group}\label{Pre1}
In this subsection, we recall several facts proved in \cite{ITGeomHW}. 
Let $q$ be a power of a prime number $p$. 
For a positive integer $m$ prime to $p$, 
we put 
\[
 \mu_m = 
 \{ a \in \ol{\bF}_q \mid a^m=1 \}. 
\]
Let $n$ be a positive integer.
Let $\oU_n$ be the unitary group over $\bF_q$ 
defined by 
the hermitian form 
\[
 \mathbb{F}_{q^2}^n \times \mathbb{F}_{q^2}^n
 \to \mathbb{F}_{q^2};\ 
 ((x_i),(x'_i)) \mapsto  \sum_{i=1}^n x_i^q x'_i. 
\]
We consider the Fermat hypersurface $S_n$
defined by 
$\sum_{i=1}^n x_i^{q+1}=0$
in $\mathbb{P}_{\mathbb{F}_q}^{n-1}$. 
Let $Y_n=\mathbb{P}_{\mathbb{F}_q}^{n-1} \setminus S_n$. 
Let $\oU_n(\bF_q)$ act on $\mathbb{P}_{\mathbb{F}_{q^2}}^{n-1}$ 
by left multiplication. 
Let $\widetilde{Y}_n$ be the affine smooth variety 
defined by $\sum_{i=1}^n x_i^{q+1}=1$ in $\mathbb{A}_{\bF_q}^n$. 
Similarly, $\oU_n(\bF_q)$ acts on $\widetilde{Y}_{n,\bF_{q^2}}$. 
The morphism 
\begin{equation}\label{torsor1}
\widetilde{Y}_n \to Y_n;\ 
(x_1,\ldots,x_n) \mapsto [x_1:\cdots:x_n]
\end{equation}
 is a $\mu_{q+1}$-torsor and $\oU_n(\bF_q)$-equivariant. 

Let $\ell \neq p$ be a prime number. 
Let $\omega_{\oU_n}$ denote 
the Weil representation of $\oU_n(\bF_q)$ over $\ol{\bQ}_{\ell}$ 
(\cf \cite[Theorem 4.9.2]{GerWeil}). 

Let $\cO$ be the ring of 
integers in an algebraic extension of $\bQ_{\ell}$. 
Let $\fm$ be the maximal ideal of $\cO$. 
We set $\mathbb{F}=\mathcal{O}/\fm$. 

Let $\Lambda \in \{ \overline{\mathbb{Q}}_{\ell}, \mathcal{O},\mathbb{F} \}$. 
For a separated and of finite type scheme $Y$ 
over $\mathbb{F}_q$ which admits a left action of a finite group $G$, 
let $G$ act on $H_{\mathrm{c}}^i(Y_{\overline{\mathbb{F}}_q},\Lambda )$ 
as $(g^\ast)^{-1}$ for $g \in G$. 
We put 
\[
 V_n = H_{\mathrm{c}}^{n-1}
 (\wt{Y}_{n,\overline{\mathbb{F}}_q},\overline{\mathbb{Q}}_{\ell}). 
\]
For $\chi \in \Hom (\mu_{q-1},\Lambda^{\times})$, 
let $\mathscr{K}_{\chi}$ denote the 
$\Lambda$-sheaf on 
$Y_{n,\bF_{q^2}}$ defined by $\chi^{-1}$ and the covering 
\eqref{torsor1}. 
For $\chi \in \mu_{q-1}^{\vee}$, we have 
$V_n[\chi] = H_{\mathrm{c}}^{n-1}
 (Y_{n,\overline{\mathbb{F}}_q},\mathscr{K}_{\chi})$, 
which is the middle degree cohomology in 
a Deligne--Lusztig induction by 
\cite[(5.1), \S 6.1]{ITGeomHW}. 

Let $X_n$ be the affine smooth variety over $\mathbb{F}_{q^2}$
defined by 
\[
 z^q+z=\sum_{i=1}^n x_i^{q+1}  
\]
in $\mathbb{A}_{\mathbb{F}_{q^2}}^{n+1}
=\Spec \mathbb{F}_{q^2}[x_1,\ldots,x_n,z]$.  
Let $\oU_n(\bF_q)$ act on $X_n$ by 
\begin{equation*}
 X_n \to X_n;\ (v,z) \mapsto (g v,z) \quad \textrm{for $g \in \oU_n(\bF_q)$}, 
\end{equation*}
where we regard $v=(x_i)$ as a column vector. 
We put 
$\bF_{q,\varepsilon} =\{ a \in \bF_{q^2} \mid a +\varepsilon a^q =0\}$. 
We sometimes abbreviate $\pm 1$ as $\pm$. 
Let $\mathbb{F}_{q,+}$ act on $X_n$ by 
$z \mapsto z+a$ for $a \in \mathbb{F}_{q,+}$. 

Let 
$\psi \in \Hom (\bF_{q,\varepsilon},\Lambda^{\times}) \setminus \{1\}$.  
Let $\mathscr{L}_{\psi}$ denote the $\Lambda$-sheaf associated to
$\psi^{-1}$ and $z^q+\varepsilon z=t$ on 
$\mathbb{A}^1_{\mathbb{F}_q}=\Spec\mathbb{F}_q[t]$. 
We consider the morphism 
\begin{equation*}
\pi \colon \mathbb{A}_{\mathbb{F}_q}^n 
\to \mathbb{A}_{\mathbb{F}_q}^1;\
(x_i)_{1 \leq i \leq n} \mapsto \sum_{i=1}^n x_i^{q+1}. 
\end{equation*}
Then we have an isomorphism 
\begin{equation}\label{XLpsi}
 H_{\mathrm{c}}^i(X_{n,\ol{\bF}_q},\Lambda)[\psi] 
 \simeq 
 H_{\mathrm{c}}^i(\bA^n,\pi^\ast \mathscr{L}_{\psi}) 
\end{equation}
for $i \geq 0$. 

\begin{prop}\label{prop:Vn}
Assume that $n \geq 2$. 
\begin{enumerate}
\item\label{en:Vomega}
We have 
$V_n \simeq \omega_{\oU_n}$ as representations of 
$\oU_n(\bF_q)$. 
\item\label{en:Vdim}
For $\chi \in \mu_{q+1}^{\vee}$, we have 
\begin{align*}
\dim V_n[\chi]
=\begin{cases}
\displaystyle \frac{q^n+(-1)^n q}{q+1} & \textrm{if $\chi=1$}, \\
\displaystyle \frac{q^n-(-1)^n}{q+1} & \textrm{if $\chi \neq 1$}. 
\end{cases}
\end{align*}
The $\oU_n(\bF_q)$-representations 
$\{V_n[\chi]\}_{\chi \in \mu_{q+1}^{\vee}}$ are irreducible and distinct. Moreover, only 
$V_n[1]$ is unipotent 
as a $\oU_n(\bF_q)$-representation. 
\end{enumerate}
\end{prop}
\begin{proof}
We have 
\begin{equation}\label{eq:VnA}
 V_n \simeq \bigoplus_{\chi \in \mu_{q+1}^{\vee}} 
 H_{\mathrm{c}}^{n-1}
 (Y_{n,\overline{\mathbb{F}}_q},\mathscr{K}_{\chi}) 
 \simeq 
 H_{\mathrm{c}}^n (\mathbb{A}^n,\pi^\ast \mathscr{L}_{\psi}) 
\end{equation}
as representations of $\oU_n(\bF_q)$ 
by \cite[Lemma 4.3, Corollary 4.6, Lemma 4.7 (2)]{ITGeomHW}. 
The claim \ref{en:Vomega} follows from 
\eqref{XLpsi}, \eqref{eq:VnA} and 
\cite[(2.6), Theorem 2.5]{ITGeomHW}. 
The claim \ref{en:Vdim} follows from the claim \ref{en:Vomega}, 
\eqref{eq:VnA} and 
\cite[Lemma 4.2, Corollary 6.2]{ITGeomHW}. 
\end{proof}

\subsection{General facts on \'{e}tale cohomology}
\label{Pre3}
We recall a basic fact on cohomology of an affine smooth variety, which 
will be used frequently. 
\begin{lem}\label{affine}
Let $X$ be an affine smooth variety over 
$\overline{\mathbb{F}}_q$ of dimension 
$d$. Let $\ell \neq p$. 
Let $F$ be a finite extension of 
$\mathbb{Q}_{\ell}$. Let $\mathcal{O}_F$
be the ring of integers of $F$. Let $\kappa_F$ be the residue field of
$\mathcal{O}_F$. 
Let $\Lambda \in \{\mathcal{O}_F,\kappa_F\}$. 
Suppose that $\mathscr{F}$ is 
a smooth $\Lambda$-sheaf 
on $X$. 
\begin{enumerate}
\item\label{en:affv}
Assume $\Lambda=\kappa_F$. Then we have 
$H_{\mathrm{c}}^i(X,\mathscr{F})=0$ for 
$i<d$. 
\item\label{en:frmid}
Assume $\Lambda=\mathcal{O}_F$. 
The middle cohomology 
$H_{\mathrm{c}}^d(X,\mathscr{F})$ is a finitely generated free  
$\mathcal{O}_F$-module. 
\end{enumerate}
\end{lem}
\begin{proof}
The first assertion follows from 
affine vanishing and Poincar\'{e} duality. 
We show the second claim.  
We take a uniformizer $\varpi$ of $\mathcal{O}_F$. 
Then, we have an exact sequence 
\[
H_{\mathrm{c}}^{d-1}(X,\mathscr{F}/\varpi)
\to H_{\mathrm{c}}^{d}(X,\mathscr{F}) \xrightarrow{\varpi} 
H_{\mathrm{c}}^{d}(X,\mathscr{F}). 
\]
Since we have $H_{\mathrm{c}}^{d-1}(X,\mathscr{F}/\varpi)=0$ 
by the first claim, the $\varpi$-multiplication map is injective. 
Since $H_{\mathrm{c}}^{d}(X,\mathscr{F})$ is a finitely generated $\mathcal{O}_F$-module, 
the second claim follows. 
\end{proof}
We recall a well-known fact, which will be used in the proof of Proposition \ref{sol}. 
\begin{lem}\label{affine2}
Let the notation be as in Lemma \ref{affine}. 
Let $G$ be a finite group. 
Let $X \to Y$ be a $G$-torsor between $d$-dimensional affine smooth 
varieties over $\overline{\mathbb{F}}_q$. 
\begin{enumerate}
\item\label{en:Gmid}
We have an isomorphism 
$H_{\mathrm{c}}^d(X,\kappa_F)^G \simeq 
H_{\mathrm{c}}^d(Y,\kappa_F)$.
\item\label{en:Ginv}
Assume $\ell \nmid \lvert G \rvert$. 
Then we have an isomorphism 
$H_{\mathrm{c}}^i(X,\kappa_F)^G \simeq H_{\mathrm{c}}^i(Y,\kappa_F)$ for any $i$.
\end{enumerate}
\end{lem}
\begin{proof}
As in \cite[Lemma 2.2]{IllBrEP}, 
we have a spectral sequence 
\[
E_2^{p,q}=H^p(G,H^q_{\mathrm{c}}(X,\kappa_F)) \Longrightarrow E^{p+q}
=H^{p+q}_{\mathrm{c}}(Y,\kappa_F). 
\]
Since we have $H^i_{\mathrm{c}}(X,\kappa_F)=0$ for $i<d$ 
by Lemma \ref{affine} \ref{en:affv}, 
we have an isomorphism $E_2^{0,d} \simeq E^d$. 
Hence the first claim follows. 

We show the second claim. 
For any $\kappa_F[G]$-module $M$, we have 
$H^i(G,M)=0$ for any $i>0$ by $\ell \nmid \lvert G \rvert$. 
Hence the second claim follows from the above spectral sequence. 
\end{proof}

\subsection{Lusztig series and $\ell$-blocks}\label{ssec:Lusbl}
We briefly recall several facts on 
Lusztig series (\cf \cite[\S7]{LusIrrcl}). 
We mainly follow \cite[Chapter 13]{DiMiRepLie}. 

Let $G$ be a connected reductive group over $\bF_q$. 
Let  $\mathrm{Irr}(G(\bF_q))$ denote the set 
of irreducible characters of $G(\bF_q)$ over $\ol{\bQ}_{\ell}$. 
Let $G^\ast$ be a connected reductive group over $\bF_q$ 
which is the dual of $G$ 
in the sense of \cite[13.10 Definition]{DiMiRepLie}. 

We fix an isomorphism 
$\ol{\bF}_q^{\times} \simeq (\bQ/\bZ)_{p'}$ 
and an embedding 
$\ol{\bF}_q^{\times} \hookrightarrow \ol{\bQ}_{\ell}^{\times}$, 
where $(\bQ/\bZ)_{p'}$ denotes the subgroup of 
$\bQ/\bZ$ consisting of the elements of order prime to $p$. 
Let $(s)$ be a geometric conjugacy class 
of a semisimple element $s \in G^{\ast}(\bF_q)$. 
As in \cite[13.16 Definition]{DiMiRepLie}, 
let $\mathcal{E}(G,(s))$ be the subset of 
$\mathrm{Irr}(G(\bF_q))$ which 
consists of irreducible constituents of a Deligne--Lusztig 
character $R^G_T(\theta)$, where 
$(T,\theta)$ is of the geometric conjugacy class 
associated to $(s)$ in the sense of 
\cite[13.2 Definition, 13.12 Proposition]{DiMiRepLie}. 
The subset $\mathcal{E}(G,(s))$ 
is called a Lusztig series associated to $(s)$.  
By \cite[13.17 Proposition]{DiMiRepLie}, 
$\mathrm{Irr}(G(\bF_q))$ is partitioned into 
Lusztig series. 
The following fact is well-known. 

\begin{lem}\label{lem:compE}
Let $f \colon G \to G'$ be a morphism between 
connected reductive groups over $\bF_q$ with 
a central connected kernel 
such that the image of $f$ contains the derived group of $G'$. 
Let $G^*$ and $G'^*$ be the dual groups of $G$ and $G'$. 
Let $s \in G^*(\bF_q)$ be the image of a semisimple element 
$s'$ in $G'^*(\bF_q)$. 
Then the irreducible constituents of 
the inflations under $f$ of elements in $\cE (G',(s'))$ 
are in $\cE (G,(s))$. 
\end{lem}
\begin{proof}
This follows from \cite[13.22 Proposition]{DiMiRepLie}. 
\end{proof}

For a semisimple $\ell'$-element 
$s \in G^\ast (\bF_q)$, we define 
\begin{equation*}
\mathcal{E}_{\ell}(G,(s))=\bigcup_{t \in 
(C_{G^\ast (\bF_q)}(s))_{\ell}} \mathcal{E}(G,(st)). 
\end{equation*}
It is known that this set is a union of $\ell$-blocks 
by \cite[2.2 Th\'eor\`eme]{BrMiBlL}. 
Any block contained in $\mathcal{E}_{\ell}(G,(1))$
is called a unipotent $\ell$-block. 

\begin{lem}\label{lem:diflb}
Let $s$ and $s'$ be semisimple $\ell'$-elements of 
$G^\ast (\bF_q)$. 
Assume that $s$ and $s'$ are not geometrically conjugate. 
Let $\rho \in \cE (G,(s))$ and $\rho' \in \cE (G,(s'))$. 
Then $\rho \notin \cE_{\ell} (G,(s'))$. 
In particular, $\rho$ and $\rho'$ are in different $\ell$-blocks. 
\end{lem}
\begin{proof}
Assume $\rho \in \cE_{\ell} (G,(s'))$.  
Then there exists an element 
$s'' \in C_{G^{\ast}(\bF_q)}(s')$ 
of $\ell$-power order such that 
$\rho \in \cE(G,(s's''))$. 
Since $\rho \in \cE(G,(s))$, 
the elements $s$ and $s' s''$ are geometrically conjugate 
by \cite[13.17 Proposition]{DiMiRepLie}. 
By the assumption that $s$ and $s'$ are not geometrically conjugate, $s'' \neq 1$. Hence the order of $s$ and $s's''$ are different because $s$ and $s'$ are $\ell'$-elements and $s'' \in C_{G^{\ast}(\bF_q)}(s')$. This is a contradiction. 
\end{proof}

For an irreducible representation $\pi$ of a finite group, 
let $\overline{\pi}$ denote the Brauer character associated to a mod $\ell$ reduction of $\pi$. 
For any integer $m \geq 1$, 
let $m_p$ be the largest power of $p$ dividing $m$ 
and $m_{p'}=m/m_p$. 

\begin{lem}\label{lem:lirr}
Let $\rho$ be an irreducible ordinary character of $G(\bF_q)$. 
Assume $\rho \in \mathcal{E}_{\ell}(G,(s))$ for 
some semisimple $\ell'$-element $s$ of $G^\ast (\bF_q)$. 
If 
\[
 \frac{\lvert G(\bF_q) \rvert_{p'}}{\lvert C_{G^\ast (\bF_q)}(s) \rvert_{p'}}=\rho(1),
\]
then $\overline{\rho}$ is an irreducible Brauer character. 
\end{lem}
\begin{proof}
Assume that $\overline{\rho}$ is not irreducible and 
take an irreducible Brauer subcharacter $\chi$ of $\overline{\rho}$. 
Then 
\[
 \frac{\lvert G(\bF_q) \rvert_{p'}}{\lvert C_{G^\ast (\bF_q)}(s) \rvert_{p'}} 
 \leq \chi(1) < \rho (1) 
\]
by \cite[Proposition 1]{HiMaLowuni}. 
This is a contradiction. 
\end{proof}

We recall some results in \cite[\S 6]{HiMaLowuni}. 
Take a nonisotropic vector $c$ in the standard representation of 
$\oU_n(\bF_q)$. 
Let $\hat{S}$ be the subgroup of $\oU_n(\bF_q)$ 
fixing the line $\langle c \rangle$ and inducing the identity 
on the orthogonal complement of $\langle c \rangle$. 
Let $S$ be the image of $\hat{S}$ in $\PU_n(\bF_q)$. 
Then $\hat{S}$ and $S$ are cyclic groups of order $q+1$. 
We follow \cite[\S 6]{HiMaLowuni} for a definition of 
the Weil characters of $\SU_n(\bF_q)$. 

\begin{lem}\label{lem:SUmod}
Assume that $n \geq 3$. 
The Weil characters of $\SU_n(\bF_q)$ 
consist of 
one unipotent character $\chi_{(n-1,1)} \in \cE (\SU_n,(1))$ 
of degree 
$(q^n +(-1)^n q)/(q+1)$ and 
$q$ non-unipotent characters 
$\chi_{s,(n-1)} \in \cE (\SU_n,(s))$ of degree 
$(q^n -(-1)^n)/(q+1)$ for $s \in S \setminus \{ 1\}$, 
where the elements of $S$ give different geometrically conjugacy classes. 
We put 
\[
 N=\min \left\{ \frac{q^n +(-1)^n q}{q+1} , \frac{q^n -(-1)^n}{q+1} \right\} 
\]
Let $V$ be a Weil character of $\SU_n(\bF_q)$. 
If the degree of $V$ is $N$, then $\ol{V}$ is irreducible. 
If the degree of $V$ is $N+1$, then $\ol{V}$ is irreducible 
or has two irreducible constituents, one of which is trivial. 
Further we have the following: 
\begin{enumerate}
\item 
Let $n \geq 4$ be even. 
Then $\ol{\chi}_{(n-1,1)}$ is irreducible if and only if 
$\ell \nmid q+1$. 
\item 
Let $n \geq 3$ be odd. 
Then $\ol{\chi}_{s,(n-1)}$ is irreducible if and only if 
the order of $s$ is not a power of $\ell$. 
\end{enumerate}
\end{lem}
\begin{proof}
Everything is explained in \cite[p.~755]{HiMaLowuni} 
except that the elements of 
$S$ give different geometrically conjugacy classes. 
Assume that different elements $s$ and $s'$ in $S$ are 
geometrically conjugate. 
Then their lifts $\hat{s}$ and $\hat{s}'$ in $\hat{S}$ 
are geometrically conjugate modulo the center. 
Considering the eigenvalues of $\hat{s}$ and $\hat{s}'$, 
we have a contradiction. 
\end{proof}

\begin{lem}\label{lem:chis}
Let $\chi \in \mu_{q+1}^{\vee} \setminus \{ 1 \}$. 
We view $\chi$ as a character of the diagonal torus 
$\oU_1(\bF_q)^n$ of $\oU_n (\bF_q)$ under the first projection. 
Let $\hat{s} \in \hat{S}$ be the element corresponding to 
$\chi$ by \cite[13.12 Proposition]{DiMiRepLie}. 
Let $s \in S$ be the image of $\hat{s}$. 
Then we have 
$V_n[1]|_{\SU_n(\bF_q)} =\chi_{(n-1,1)}$ 
and 
$V_n[\chi]|_{\SU_n(\bF_q)} =\chi_{s,(n-1)}$. 
\end{lem}
\begin{proof}
This follows from Lemma \ref{lem:compE} and Lemma \ref{lem:SUmod}. 
\end{proof}

\section{Cohomology as representation of unitary group}\label{mod1}
In this section, we investigate mainly 
$H_{\mathrm{c}}^n(Y_{n,\overline{\mathbb{F}}_q},\ol{\bF}_{\ell})$ 
as an $\ol{\bF}_{\ell}[\oU_n(\bF_q)]$-module. 

In this section, we often ignore Tate twists 
when it is not necessary to consider Frobenius action. 
We recall that $\mathcal{O}$ and $\mathbb{F}$ is defined in \S \ref{Pre1}. 
For an $\mathcal{O}$-module $M$, 
let $M[\fm]$ denote the $\mathcal{O}$-submodule of $M$
consisting of elements  
annihilated by any element of $\fm$. 

\begin{lem}\label{free}
\begin{enumerate}
\item\label{en:fr}
The cohomology 
$H^i(S_{n,\overline{\mathbb{F}}_q},\mathcal{O})$
is free as an $\mathcal{O}$-module for any $i$.
\item\label{en:rk1}
The cohomology $H^i(S_{n,\overline{\mathbb{F}}_q},\mathcal{O})$ is zero if $i \neq n-2$ and 
$i$ is odd, and is free of rank one as an $\mathcal{O}$-module if $0 \leq i \leq 2(n-2)$ is even and $i \neq n-2$.  
\item 
We have an isomorphism 
$H^i(S_{n,\overline{\mathbb{F}}_q},\mathcal{O})\otimes_{\mathcal{O}} \mathbb{F}\simeq H^i(S_{n,\overline{\mathbb{F}}_q},\mathbb{F})$ for any $i$. 
\end{enumerate}
\end{lem}
\begin{proof}
We denote by $S_{n,\mathbb{Q}}$ 
the Fermat variety defined by 
the same equation as $S_n$ in $\mathbb{P}_{\mathbb{Q}}^{n-1}$.  
Let $i$ be an integer. 
We have isomorphisms
\[
H^i(S_{n,\mathbb{C}}^{\rm an},\mathbb{Z}) \otimes_{\mathbb{Z}} \mathcal{O} \simeq 
H^i(S_{n,\mathbb{C}}^{\rm an},\mathcal{O}) \simeq
H^i(S_{n,\mathbb{C}},\mathcal{O}),  
\]
where the first isomorphism follows from that 
$\cO$ is flat over $\bZ$, and  
the second one follows from the comparison theorem 
between singular and \'{e}tale cohomology. 
By taking an isomorphism 
$\mathbb{C} \simeq \overline{\mathbb{Q}}_p$, 
we have an isomorphism $H^i(S_{n,\mathbb{C}},\mathcal{O}) \simeq H^i(S_{n,\overline{\mathbb{Q}}_p},\mathcal{O})$. Since $S_{n,\mathbb{Q}}$ has good reduction at $p$ and the reduction equals $S_n$, we have an isomorphism 
\[
H^i(S_{n,\overline{\mathbb{Q}}_p},\mathcal{O}) \simeq 
H^i(S_{n,\overline{\mathbb{F}}_q},\mathcal{O})
\]
by the proper base change theorem. 
As a result, we have 
\begin{equation}\label{comparison}
H^i(S_{n,\mathbb{C}}^{\rm an},\mathbb{Z}) \otimes_{\mathbb{Z}} \mathcal{O} \simeq 
H^i(S_{n,\overline{\mathbb{F}}_q},\mathcal{O}). 
\end{equation}
Hence the first claim follows, because $H^i(S_{n,\mathbb{C}}^{\rm an},\mathbb{Z})$ is a free $\mathbb{Z}$-module by \cite[Proposition (B32) (ii)]{DimSinghyp}. 
The second claim follows from \eqref{comparison} and \cite[Theorem (B22)]{DimSinghyp}. 
We have a short exact sequence \[
0 \to H^i(S_{n,\overline{\mathbb{F}}_q},\mathcal{O})\otimes_{\mathcal{O}} \mathbb{F} \to H^i(S_{n,\overline{\mathbb{F}}_q},\mathbb{F}) \to H^{i+1}(S_{n,\overline{\mathbb{F}}_q},\mathcal{O})[\fm] \to 0. 
\] Hence the third claim follows from the first one. 
\end{proof}
In the sequel, we always assume that $\cO$ is the ring of integers in 
a finite extension of $\mathbb{Q}_{\ell}(\mu_{p(q+1)})$. 
Every homomorphism is $\oU_n(\bF_q)$-equivariant. Let the notation be as in 
\S \ref{agd}. 
We have a long exact sequence 
\begin{gather}\label{ltri}
\begin{aligned}
\cdots & \to H_{\mathrm{c}}^i(Y_{n,\overline{\mathbb{F}}_q},\mathcal{O}) \to 
H^i(\mathbb{P}^{n-1},\mathcal{O})
\to 
H^i(S_{n,\overline{\mathbb{F}}_q},\mathcal{O})
\\
& \to H_{\mathrm{c}}^{i+1}(Y_{n,\overline{\mathbb{F}}_q},\mathcal{O}) \to 
H^{i+1}(\mathbb{P}^{n-1},\mathcal{O})
\to 
H^{i+1}(S_{n,\overline{\mathbb{F}}_q},\mathcal{O})
\to \cdots. 
\end{aligned}
\end{gather}
By Lemma \ref{affine} \ref{en:affv}, the restriction map 
\[
H^i(\mathbb{P}^{n-1},\mathcal{O})
\to 
H^i(S_{n,\overline{\mathbb{F}}_q},\mathcal{O})
\] 
is an isomorphism for $i<n-2$. By Lemma \ref{free} \ref{en:fr}
and Poincar\'{e} duality, 
we obtain an isomorphism
\[
f_i \colon 
H^i(S_{n,\overline{\mathbb{F}}_q},\mathcal{O}) \xrightarrow{\sim} H^{i+2}(\mathbb{P}^{n-1},\mathcal{O})
\] for $n-1 \leq i \leq 2n-4$. 
Let $\kappa=[\mathbb{P}^{n-2}] \in H^2(\mathbb{P}^{n-1},\mathcal{O})$ be the cycle class of 
the hyperplane $\mathbb{P}^{n-2} \subset \mathbb{P}^{n-1}$, and 
$\kappa^i \in H^{2i}(\mathbb{P}^{n-1},\mathcal{O})$ the cup product of $\kappa$. For $1 \leq i \leq n-1$, 
the map $\kappa^i \colon \mathcal{O} \to H^{2i}(\mathbb{P}^{n-1},\mathcal{O});\ 1 \mapsto \kappa^i$ is an isomorphism. 
Let $[S_{n,\overline{\mathbb{F}}_q} ] \in H^2(\mathbb{P}^{n-1},\mathcal{O})$
be the cycle class of $S_{n,\overline{\mathbb{F}}_q} \subset \mathbb{P}^{n-1}$. 
For $n-1 \leq 2i \leq 2n-4$, the composite 
\begin{equation*}\label{tol}
\xymatrix{
H^{2i}(\mathbb{P}^{n-1},\mathcal{O})
\ar[r]^-{\rm rest.} &
H^{2i}(S_{n,\overline{\mathbb{F}}_q},\mathcal{O})
\ar[r]^-{f_{2i}}_-{\simeq} &  
H^{2i+2}(\mathbb{P}^{n-1},\mathcal{O})
}
\end{equation*}
equals the map induced by the cup product 
by $[S_{n,\overline{\mathbb{F}}_q} ]$. 
Clearly, we have 
$[S_{n,\overline{\mathbb{F}}_q} ]=(q+1) \kappa$
in $H^2(\mathbb{P}^{n-1},\mathcal{O})$. 
We write $q+1=\ell^a r$ with $(\ell,r)=1$. 
By Lemma \ref{free} \ref{en:rk1}, we have 
\begin{equation}\label{T}
H_{\mathrm{c}}^i(Y_{n,\overline{\mathbb{F}}_q},\mathcal{O})
\simeq \begin{cases}
0 & \textrm{if $i$ is even}, \\
\mathcal{O}/\ell^a & \textrm{if $i$ is odd}
\end{cases}
\end{equation}
for $n \leq i < 2n-2$. 


For a character $\chi \colon \mu_{q+1} \to 
\mathcal{O}^{\times}$, we write $\overline{\chi}$ for the composite 
of $\chi$ and the reduction map 
$\mathcal{O}^{\times} \to \mathbb{F}^{\times}$. 

We have a short exact sequence 
\begin{equation}\label{T2}
0 \to H_{\mathrm{c}}^{n-1}(Y_{n,\overline{\mathbb{F}}_q},\mathscr{K}_{\chi})\otimes_{\mathcal{O}} \mathbb{F}
\to 
H_{\mathrm{c}}^{n-1}(Y_{n,\overline{\mathbb{F}}_q},\mathscr{K}_{\overline{\chi}}) \to 
H_{\mathrm{c}}^{n}(Y_{n,\overline{\mathbb{F}}_q}, \mathscr{K}_{\chi})[\fm] \to 0.  
\end{equation}
By \eqref{T}, we have 
\begin{equation}\label{T2-5}
\dim_{\mathbb{F}} H_{\mathrm{c}}^{n}(Y_{n,\overline{\mathbb{F}}_q},\mathcal{O})[\fm]=
\begin{cases}
0 & \textrm{if $a=0$}, \\
\frac{1-(-1)^n}{2} & \textrm{if $a \geq 1$}.
\end{cases}  
\end{equation}
By Proposition \ref{prop:Vn}, 
Lemma \ref{affine} \ref{en:frmid},  
\eqref{T2} with $\chi=1$ and \eqref{T2-5}, 
we have 
\begin{equation}\label{ddm}
\dim_{\mathbb{F}}
H_{\mathrm{c}}^{n-1}(Y_{n,\overline{\mathbb{F}}_q},\mathbb{F})=
\begin{cases}
\frac{q^n+(-1)^nq}{q+1} & \textrm{if $a=0$}, \\
\frac{q^n-(-1)^n}{q+1}+\frac{1+(-1)^n}{2} & \textrm{if $a\geq 1$}. 
\end{cases}
\end{equation} 
Hence, we have 
\begin{equation}\label{T3}
\dim_{\mathbb{F}} H_{\mathrm{c}}^{n}(Y_{n,\overline{\mathbb{F}}_q},\mathscr{K}_{\chi})[\fm]=\frac{1+(-1)^n}{2} \quad \textrm{if $\chi \neq 1$ and $\overline{\chi}=1$}
\end{equation}
again by Proposition \ref{prop:Vn}, 
Lemma \ref{affine} \ref{en:frmid} and   
\eqref{T2} with $\chi$. 
Note that a character $\chi \neq 1$ such that $\overline{\chi}=1$ does not exist if $a=0$. 

In the following, we investigate \eqref{T2} when $\overline{\chi} \neq 1$. 
We set $Y_{n,r}=\widetilde{Y}_{n,\overline{\mathbb{F}}_q}/\mu_{\ell^a}$. 
We have a decomposition 
$\mu_{q+1} = \mu_{\ell^a} \times \mu_r$. 
We write as $\chi=\chi_{\ell^a} \chi_r$
with $\chi_{\ell^a} \in \Hom (\mu_{\ell^a},\mathcal{O}^{\times})$
and $\chi_r \in \Hom (\mu_r,\mathcal{O}^{\times})$.  
We have 
\begin{equation}\label{T4}
0 \to H_{\mathrm{c}}^{n-1}(Y_{n,r},\mathscr{K}_{\chi_{\ell^a}})\otimes_{\mathcal{O}} \mathbb{F}
\to 
H_{\mathrm{c}}^{n-1}(Y_{n,r},\mathbb{F}) \to 
H_{\mathrm{c}}^{n}(Y_{n,r},{\mathscr{K}_{\chi_{\ell^a}}})[\fm] \to 0.  
\end{equation}
The natural morphism $Y_{n,r} \to Y_{n,\overline{\mathbb{F}}_q}$ is a $\mu_r$-torsor. 
By $(\ell,r)=1$, we have isomorphisms 
\begin{gather}\label{T4-5}
\begin{aligned}
H_{\mathrm{c}}^i(Y_{n,r},\mathscr{K}_{\chi_{\ell^a}}) & \simeq 
\bigoplus_{\chi_r \in \Hom (\mu_r,\mathcal{O}^{\times})} 
H_{\mathrm{c}}^i(Y_{n,\overline{\mathbb{F}}_q},\mathscr{K}_{\chi_{\ell^a} \chi_r}), \\
H_{\mathrm{c}}^i(Y_{n,r},\mathbb{F}) & \simeq \bigoplus_{\ol{\chi}_r \in \Hom (\mu_r,\mathbb{F}^{\times})}
H_{\mathrm{c}}^i(Y_{n,\overline{\mathbb{F}}_q},\mathscr{K}_{\ol{\chi}_r})
\end{aligned}
\end{gather}
for any integer $i$. 
By Proposition \ref{prop:Vn}, Lemma \ref{affine} \ref{en:frmid} 
and the first isomorphism in \eqref{T4-5}, we have 
\begin{equation}\label{T5}
\mathrm{rank}_{\mathcal{O}}\ H_{\mathrm{c}}^{n-1}(Y_{n,r},\mathscr{K}_{\chi_{\ell^a}})=
\begin{cases}
 \frac{q^n-(-1)^n}{q+1} r & \textrm{if $\chi_{\ell^a} \neq 1$}, \\
\frac{q^n-(-1)^n}{q+1}r +(-1)^n & 
\textrm{if $\chi_{\ell^a}=1$}. 
\end{cases}
\end{equation}
We show the following lemma by using a comparison theorem between singular and \'{e}tale cohomology and applying results on weighted hypersurfaces in \cite{DimSinghyp}. 
\begin{lem}\label{surj}
The pull-back $H_{\mathrm{c}}^i(Y_{n,\overline{\mathbb{F}}_q},\mathcal{O}) \to 
H_{\mathrm{c}}^i(Y_{n,r},\mathcal{O})$ is an isomorphism
for $n \leq i <2n-2$. 
\end{lem}
\begin{proof}
Let $f \colon Y_{n,r} \to Y_{n,\overline{\mathbb{F}}_q}$
be the natural finite morphism. 
We have the trace map
$f_{\ast} \colon H_{\mathrm{c}}^i(Y_{n,r},\mathcal{O}) \to 
H_{\mathrm{c}}^i(Y_{n,\overline{\mathbb{F}}_q},\mathcal{O})$.  
The pull-back map 
$f^\ast \colon H_{\mathrm{c}}^i(Y_{n,\overline{\mathbb{F}}_q},\mathcal{O}) 
 \to H_{\mathrm{c}}^i(Y_{n,r},\mathcal{O})$ is injective by 
$(\ell,r)=1$, because 
the composite $f_{\ast} \circ f^\ast$ is the 
$r$-multiplication map. 
Therefore it suffices to show that the map is surjective. 

We regard $S_n$ as a closed 
subscheme of $S_{n+1}$ defined by $x_{n+1}=0$. 
We take $\xi \in \mathbb{F}_{q^2}$ such that 
$\xi^{q+1}=-1$. 
Then $\widetilde{Y}_{n}$ is isomorphic to the complement 
$S_{n+1} \setminus S_{n}$ over $\bF_{q^2}$ by 
\[
 \widetilde{Y}_{n,\bF_{q^2}} \xrightarrow{\sim} 
 (S_{n+1} \setminus S_{n})_{\bF_{q^2}};\ 
 (x_i)_{1 \leq i \leq n} \mapsto [x_1:\cdots:x_n:\xi]. 
\]
Let $\mu_{q+1}$ act on $S_{n+1,\bF_{q^2}}$ by 
$[x_1:\cdots:x_{n+1}] \mapsto [x_1:\cdots :x_n:\zeta^{-1}x_{n+1}]$ for $\zeta \in \mu_{q+1}$. 
We have the well-defined morphism 
$\pi \colon S_{n+1} \to \mathbb{P}_{\mathbb{F}_{q^2}}^{n-1};\ 
[x_1:\cdots:x_{n+1}] \mapsto [x_1:\cdots:x_n]$. 
We have a commutative diagram
\[
\xymatrix{
\widetilde{Y}_{n,\bF_{q^2}} \ar[r]\ar[d] &  S_{n+1,\bF_{q^2}} \ar[d]^{\pi} & S_{n,\bF_{q^2}} \ar[l]\ar[d]^{\simeq} \\
Y_{n,\bF_{q^2}} \ar[r] & \mathbb{P}_{\mathbb{F}_{q^2}}^{n-1} & S_{n,\bF_{q^2}}.  \ar[l] 
}
\]
By considering the base change of this to 
$\overline{\mathbb{F}}_q$ and 
taking the quotients on the upper line by $\mu_{\ell^a}$, 
we obtain a commutative diagram 
\begin{equation}\label{comm}
\xymatrix{
Y_{n,r} \ar[r]\ar[d]^f & S_{n+1,{\overline{\mathbb{F}}_q}}/\mu_{\ell^a}\ar[d] & S_{n,{\overline{\mathbb{F}}_q}}
\ar[l]\ar[d]^{\simeq} \\
Y_{n,{\overline{\mathbb{F}}}_q} \ar[r] & \mathbb{P}^{n-1} & S_{n,{\overline{\mathbb{F}}_q}}. \ar[l]
}
\end{equation}
Let $\mathbf{w}=(1,\ldots,1,\ell^a) \in \mathbb{Z}_{\geq 1}^{n+1}$ and $\mathbb{P}(\mathbf{w})$
be the weighted projective space associated to 
$\mathbf{w}$ over $\overline{\mathbb{F}}_q$. 
Then  the quotient $S_{n+1,\overline{\mathbb{F}}_q}/\mu_{\ell^a}$ 
is isomorphic to the weighted hypersurface defined by 
\begin{equation}\label{sa}
\sum_{i=1}^n X_i^{q+1}+X_{n+1}^r=0
\end{equation}
in $\mathbb{P}(\mathbf{w}) \simeq \mathbb{P}^n/\mu_{\ell^a}$. 
Let $U_i \subset S_{n+1,\overline{\mathbb{F}}_q}$ be the open subscheme defined by $x_i \neq 0$ for $1 \leq i \leq n$. Then we have $S_{n+1,\overline{\mathbb{F}}_q}=\bigcup_{i=1}^n U_i$. 
For each $1 \leq i \leq n$, the quotient $U_i/\mu_{\ell^a}$ is defined by 
\[
 1+s_1^{q+1}+\cdots+s_{i-1}^{q+1}+ s_{i+1}^{q+1}+\cdots+s_n^{q+1}+t_i^r=0
\] in $\mathbb{A}^n$. This is smooth over $\overline{\mathbb{F}}_q$
by $p \nmid q+1$ and the Jacobian criterion. 
Hence $S_{n+1,\overline{\mathbb{F}}_q}/\mu_{\ell^a}$
is smooth over $\overline{\mathbb{F}}_q$. 

We consider the smooth hypersurface defined by the same equation as \eqref{sa} in the weighted projective space $\mathbb{P}(\mathbf{w})$ over 
$\mathbb{Q}$, which we denote by $S'$. Then $S'^{\rm an}_{\mathbb{C}}$ is strongly smooth as in \cite[Example (B31)]{DimSinghyp}. 
Hence the integral cohomology algebra of it is torsion-free by 
\cite[Proposition (B32) (ii)]{DimSinghyp}. 
Clearly, $S'$ has good reduction at $p$, and the reduction is isomorphic to $S_{n+1,\overline{\mathbb{F}}_q}/\mu_{\ell^a}$ over $\overline{\mathbb{F}}_q$. 
In the same manner as \eqref{comparison}, we have an isomorphism 
\[
H^i(S'^{\rm an}_{\mathbb{C}},\mathbb{Z}) \otimes_{\mathbb{Z}} \mathcal{O} \simeq H^i(S_{n+1,{\overline{\mathbb{F}}_q}}/\mu_{\ell^a},\mathcal{O}).
\] 
Hence, $H^i(S_{n+1,{\overline{\mathbb{F}}_q}}/\mu_{\ell^a},\mathcal{O})$ is a free $\mathcal{O}$-module of rank one for any even integer $i \neq n-1$ and 
$H^i(S_{n+1,{\overline{\mathbb{F}}_q}}/\mu_{\ell^a},\mathcal{O})=0$ for any odd 
integer $i \neq n-1$
by \cite[(B33)]{DimSinghyp}. 
By \eqref{comm}, we have a commutative diagram
\[
\xymatrix
{
H^{2i}(\mathbb{P}^{n-1},\mathcal{O}) \ar[r]\ar[d]
& H^{2i}(S_{n,\overline{\mathbb{F}}_q},\mathcal{O}) \ar[r]\ar[d]^{\simeq} & H_{\mathrm{c}}^{2i+1}(Y_{n,\overline{\mathbb{F}}_q},\mathcal{O}) \ar[r]\ar[d] & 0 \\
H^{2i}(S_{n+1,{\overline{\mathbb{F}}_q}}/\mu_{\ell^a},\mathcal{O}) \ar[r]
& H^{2i}(S_{n,\overline{\mathbb{F}}_q},\mathcal{O}) \ar[r] & H_{\mathrm{c}}^{2i+1}(Y_{n,r},\mathcal{O}) \ar[r] & 0}
\]
for $n-1 \leq 2i <2n-3$, where the horizontal 
lines are exact. Hence the right vertical map is surjective. 
Hence the claim for any odd integer $i$
follows.  

By \eqref{T}, it suffices to show 
$H_{\mathrm{c}}^{2i}(Y_{n,r},\mathcal{O})=0$
for $n \leq 2i <2n-2$. 
We have an exact sequence 
\[
0 \to H_{\mathrm{c}}^{2i}(Y_{n,r},\mathcal{O}) \xrightarrow{g_i} H^{2i}(S_{n+1,{\overline{\mathbb{F}}_q}}/\mu_{\ell^a},\mathcal{O}) \to H^{2i}(S_{n,\overline{\mathbb{F}}_q},\mathcal{O}) 
\]
for $n \leq 2i<2n-2$ by Lemma \ref{free} \ref{en:rk1}. If $H_{\mathrm{c}}^{2i}(Y_{n,r},\mathcal{O}) \neq 0$, the cokernel of $g_i$ is torsion, 
because $H^{2i}(S_{n+1,{\overline{\mathbb{F}}_q}}/\mu_{\ell^a},\mathcal{O})$ is a free $\mathcal{O}$-module of rank one.  Since $H^{2i}(S_{n,\overline{\mathbb{F}}_q},\mathcal{O})$ is a free
$\mathcal{O}$-module by Lemma \ref{free} \ref{en:fr}, we obtain $H_{\mathrm{c}}^{2i}(Y_{n,r},\mathcal{O}) =0$. 
\end{proof}
We show a fundamental proposition through the paper. 
\begin{prop}\label{lp}
\begin{enumerate}
\item\label{en:KFnd} 
Assume $\ell \nmid q+1$. 
Let $\chi \in \Hom (\mu_{q+1},\mathcal{O}^{\times})$. 
We have an isomorphism 
\[
H_{\mathrm{c}}^{n-1}(Y_{n,\overline{\mathbb{F}}_q},\mathscr{K}_{\chi})\otimes_{\mathcal{O}} \mathbb{F}
\xrightarrow{\sim}
H_{\mathrm{c}}^{n-1}(Y_{n,\overline{\mathbb{F}}_q},\mathscr{K}_{\overline{\chi}})
\]
as $\mathbb{F}[\oU_n(\bF_q)]$-modules. 
\item\label{en:KFd}
Assume $\ell \mid q+1$. 
We have a short exact sequence 
\[
0  \to H_{\mathrm{c}}^{n-1}(Y_{n,\overline{\mathbb{F}}_q},\mathscr{K}_{\chi_{\ell^a} \chi_r})\otimes_{\mathcal{O}} \mathbb{F}
\to 
H_{\mathrm{c}}^{n-1}(Y_{n,\overline{\mathbb{F}}_q},\mathscr{K}_{\overline{\chi}_r}) \to 
H_{\mathrm{c}}^{n}(Y_{n,\overline{\mathbb{F}}_q},{\mathscr{K}_{\chi_{\ell^a} \chi_r}})[\fm] \to 0 
\]
as $\mathbb{F}[\oU_n(\bF_q)]$-modules. 
Furthermore, we have 
\[
\dim_{\mathbb{F}} H_{\mathrm{c}}^{n}(Y_{n,\overline{\mathbb{F}}_q},{\mathscr{K}_{\chi_{\ell^a}\chi_r}})[\fm]
=\begin{cases}
0 & \textrm{if $\chi_r \neq 1$}, \\
\frac{1+(-1)^n}{2} & \textrm{if $\chi_r=1$ and $\chi_{\ell^a} \neq 1$}, \\
\frac{1-(-1)^n}{2} & \textrm{if $\chi_r=1$ and $\chi_{\ell^a}=1$}. 
\end{cases}
\]
\end{enumerate}
\end{prop}
\begin{proof}
By \eqref{T} and Lemma \ref{surj}, we have 
\begin{equation}\label{T6}
H_{\mathrm{c}}^i(Y_{n,r},\mathcal{O}) \xleftarrow{\sim} 
H_{\mathrm{c}}^i(Y_{n,\overline{\mathbb{F}}_q},\mathcal{O}) \simeq 
\begin{cases}
0 & \textrm{if $i$ is even}, \\
\mathcal{O}/\ell^a  & \textrm{if 
$i$ is odd}
\end{cases}
\end{equation}
for $n \leq i < 2n-2$. 
Hence, by \eqref{T4} with $\chi_{\ell^a}=1$ and \eqref{T5}, we have 
\[
\dim_\mathbb{F} H_{\mathrm{c}}^{n-1}(Y_{n,r},\mathbb{F})=
\begin{cases}
\frac{q^n-(-1)^n}{q+1} r+(-1)^n & \textrm{if $a=0$}, \\
\frac{q^n-(-1)^n}{q+1} r+\frac{1+(-1)^n}{2} & \textrm{if $a \geq 1$}. 
\end{cases}
\]
Hence we have 
\[
\dim_{\mathbb{F}} H_{\mathrm{c}}^{n}(Y_{n,r},{\mathscr{K}_{\chi_{\ell^a}}})[\fm]
=\begin{cases}
0 & \textrm{if $a=0$}, \\
\frac{1+(-1)^n}{2} & \textrm{if $a \geq 1$ and $\chi_{\ell^a} \neq 1$}, \\
\frac{1-(-1)^n}{2} & \textrm{if $a \geq 1$ and $\chi_{\ell^a}=1$} 
\end{cases}
\]
by \eqref{T4} and \eqref{T5}. 
According to \eqref{T2-5} and \eqref{T3}, we have the same 
formula for $\dim_{\mathbb{F}} H_{\mathrm{c}}^{n}(Y_{n,\overline{\mathbb{F}}_q},{\mathscr{K}_{\chi_{\ell^a}}})[\fm]$. 
Hence we have 
\begin{equation}\label{lpp0}
H_{\mathrm{c}}^{n}(Y_{n,\overline{\mathbb{F}}_q},\mathscr{K}_{\chi_{\ell^a} \chi_r})[\fm]
=0 \quad \textrm{if $\chi_r \neq 1$}
\end{equation}
by \eqref{T4-5}. 
Hence the latter claim in the claim 
\ref{en:KFd} is proved. The former one is \eqref{T2}.  
By \eqref{T2} and \eqref{lpp0}, we have an isomorphism
\begin{equation}\label{lpp}
H_{\mathrm{c}}^{n-1}(Y_{n,\overline{\mathbb{F}}_q},\mathscr{K}_{\chi_{\ell^a} \chi_r})\otimes_{\mathcal{O}} \mathbb{F}
\simeq
H_{\mathrm{c}}^{n-1}(Y_{n,\overline{\mathbb{F}}_q},\mathscr{K}_{\overline{\chi}_r}) \quad \textrm{if $\chi_r \neq 1$}.  
\end{equation}

The claim \ref{en:KFnd} for $\chi=1$ follows from \eqref{T2} with 
$\chi=1$ and \eqref{T2-5}, and the one for $\chi \neq 1$ follows from \eqref{lpp}. 
\end{proof}
\begin{lem}\label{ccc0}
Assume $\ell \mid q+1$. 
\begin{enumerate}
\item 
Assume that $n$ is odd. 
Then $H_{\mathrm{c}}^n(Y_{n,\overline{\mathbb{F}}_q},\mathcal{O})[\fm]$
is a trivial representation of $\oU_n(\bF_q)$. 
\item\label{en:evtri}
Assume that $n \geq 4$ is even and $\ell \neq 2$. 
Let $\chi_{\ell^a} \in \Hom (\mu_{q+1},\mathcal{O}^{\times})$ be a non-trivial character of $\ell$-power order. Then $H_{\mathrm{c}}^n(Y_{n,\overline{\mathbb{F}}_q},\mathscr{K}_{\chi_{\ell^a}})[\fm]$ 
is a trivial representation of $\oU_n(\bF_q)$. 
\end{enumerate}
\end{lem}
\begin{proof}
Assume that $n$ is odd. 
Since $H_{\mathrm{c}}^{n-1}(S_{n,\overline{\mathbb{F}}_q},\mathcal{O})$
is a trivial $\oU_n(\bF_q)$-representation 
by \cite[Theorem 1]{HoMaTT}, 
$H_{\mathrm{c}}^n(Y_{n,\overline{\mathbb{F}}_q},\mathcal{O})$ is so by
\eqref{ltri}. Hence, the first assertion follows. 

We show the second claim. By Proposition \ref{lp} \ref{en:KFd}, 
we have isomorphisms
\begin{gather}\label{hh} 
\begin{aligned}
& H_{\mathrm{c}}^{n-1}(Y_{n,\overline{\mathbb{F}}_q},\mathcal{O}) \otimes_{\mathcal{O}}\mathbb{F} \simeq 
H_{\mathrm{c}}^{n-1}(Y_{n,\overline{\mathbb{F}}_q},\mathbb{F}), \\
& H_{\mathrm{c}}^{n-1}(Y_{n,\overline{\mathbb{F}}_q},\mathcal{O}) \otimes_{\mathcal{O}} \overline{\mathbb{Q}}_{\ell} \simeq 
H_{\mathrm{c}}^{n-1}(Y_{n,\overline{\mathbb{F}}_q},\overline{\mathbb{Q}}_{\ell}) 
\end{aligned}
\end{gather}
and 
a short exact sequence 
\[
0 \to 
H_{\mathrm{c}}^{n-1}(Y_{n,\overline{\mathbb{F}}_q},\mathscr{K}_{\chi_{\ell^a}}) \otimes_{\mathcal{O}}\mathbb{F} \to 
H_{\mathrm{c}}^{n-1}(Y_{n,\overline{\mathbb{F}}_q},\mathbb{F})  \to 
H_{\mathrm{c}}^n(Y_{n,\overline{\mathbb{F}}_q},\mathscr{K}_{\chi_{\ell^a}})[\fm] \to 0.  
\]
By these and Lemma \ref{lem:SUmod}, 
$H_{\mathrm{c}}^n(Y_{n,\overline{\mathbb{F}}_q},\mathscr{K}_{\chi_{\ell^a}})[\fm]$ 
is a trivial $\mathbb{F}[\SU_n(\bF_q)]$-module. 
Hence, the action of $\oU_n(\bF_q)$ on it factorizes through $\det$ 
by $n \geq 4$ and \cite[(1), (8) in the proof of Theorem 3.3]{GerWeil}.  
In the sequel, we need the assumption $\ell \neq 2$, because we apply \cite{FoSrBlgu}.
Let $\chi \in \mu_r^{\vee} \setminus \{1\}$.  
Then there exists a semisimple 
$\ell'$-element $s_{\chi}$ in the center of $\oU_n(\bF_q)$ such that 
the character $\chi \circ \det$ of $\oU_n(\bF_q)$ 
belongs to the $\ell$-block 
corresponding to $s_{\chi}^{\oU_n(\bF_q)}$ in the notation in  
\cite[the first paragraph of \S6]{FoSrBlgu}. 
Then $s_{\chi}$ is non-trivial by 
\cite[p.~116, Theorem (6A)]{FoSrBlgu} 
using the fact that the 
$1$-dimensional unipotent representation of 
$\oU_n(\bF_q)$ is trivial. 
Recall that $H_{\mathrm{c}}^{n-1}(Y_{n,\overline{\mathbb{F}}_q},\overline{\mathbb{Q}}_{\ell})$ is a unipotent 
$\oU_n(\bF_q)$-representation by Proposition \ref{prop:Vn}. Hence it belongs to the block corresponding to $1^{\oU_n(\bF_q)}$. 
The blocks corresponding to $s_{\chi}^{\oU_n(\bF_q)}$ and 
$1^{\oU_n(\bF_q)}$ are distinct by 
\cite[Theorem (5D)]{FoSrBlgu}. 
Hence, $\overline{\chi} \circ \det$ can not appear as a quotient of 
$H_{\mathrm{c}}^{n-1}(Y_{n,\overline{\mathbb{F}}_q},\mathbb{F})$ by \eqref{hh}. 
Therefore, the claim follows. 
\end{proof}

\begin{cor}\label{sde}
We have $\overline{V_n[\chi]}
=H_{\mathrm{c}}^{n-1}(Y_{n,\overline{\mathbb{F}}_q},\mathscr{K}_{\overline{\chi}})$ for any 
$\chi \in \mu_{q+1}^{\vee}$
if $\ell \nmid q+1$, and 
\begin{align*}
& \overline{V_n[\chi_{\ell^a}]}+\frac{1+(-1)^n}{2}=\overline{V_n[1]}+\frac{1-(-1)^n}{2}=H_{\mathrm{c}}^{n-1}(Y_{n,\overline{\mathbb{F}}_q},\ol{\bF}_{\ell}) \ 
\textrm{for any $\chi_{\ell^a} \in \mu_{\ell^a}^{\vee} \setminus \{1\}$}, \\
& \overline{V_n[\chi_{\ell^a} \chi_r]}=\overline{V_n[\chi_r]}=H_{\mathrm{c}}^{n-1}(Y_{n,\overline{\mathbb{F}}_q},\mathscr{K}_{\overline{\chi_r}}) \quad \textrm{for any $\chi_{\ell^a} \in \mu_{\ell^a}^{\vee}$ and 
$\chi_r \in \mu_r^{\vee} \setminus \{1\}$} 
\end{align*}
 if $\ell \mid q+1$ as Brauer characters of $\oU_n(\bF_q)$. 
\end{cor}
\begin{proof}
The claims follow from Proposition \ref{lp} and 
Lemma \ref{ccc0}. 
\end{proof}

We deduce the following proposition by combining  
the above theory with Corollary \ref{sde}.  
\begin{prop}\label{ccc}
We assume that $n \geq 3$. 
\begin{enumerate}
\item\label{en:irr1}
Assume $\ell \nmid q+1$. 
The $\oU_n(\bF_q)$-representations 
\[
H_{\mathrm{c}}^{n-1}(Y_{n,\overline{\mathbb{F}}_q},\mathscr{K}_{\xi})
\quad \textrm{for $\xi \in \Hom (\mu_{q+1},\ol{\bF}_{\ell}^{\times})$}
\]
are irreducible. Moreover,
these are distinct.  
\item\label{en:irr2}
Assume $\ell \mid q+1$. 
Moreover, we suppose $\ell \neq 2$ if $n$ is even. 
The middle cohomology 
$H_{\mathrm{c}}^{n-1}(Y_{n,\overline{\mathbb{F}}_q},\ol{\bF}_{\ell})$ has two 
irreducible constituents one of which is a trivial character. The $\oU_n(\bF_q)$-representations 
\[
H_{\mathrm{c}}^{n-1}(Y_{n,\overline{\mathbb{F}}_q},\mathscr{K}_{\xi}) \quad \textrm{for $\xi \in \Hom (\mu_r,\ol{\bF}_{\ell}^{\times}) \setminus \{1\}$}
\] are irreducible and distinct. 
\end{enumerate}
\end{prop}
\begin{proof}
Let $S$ be as in \S \ref{ssec:Lusbl}. 
Let $S' \subset S$ denote the subgroup of order $r$. 
We have 
\[
 \left\{ V_n[\chi_r]|_{\SU_n(\bF_q)} \mid 
 \chi_r \in \mu_r^{\vee} \setminus \{1\} \right\}
 =\left\{ \chi_{s,(n-1)} \mid s \in S' \setminus \{1\} \right\} 
\]
by Lemma \ref{lem:chis}. 
Hence all the claims other than distinction 
follow from 
Proposition \ref{prop:Vn}, 
Lemma \ref{lem:SUmod} 
and Corollary \ref{sde}. 

It remains to show that 
$\ol{\chi}_{(n-1,1)}$ and 
$\ol{\chi}_{s,(n-1)}$ for 
$s \in S' \setminus \{1\}$ are all different. 
This follows from 
Lemma \ref{lem:diflb} and Lemma \ref{lem:SUmod}. 
\end{proof}

\section{Cohomology as representation of symplectic group}\label{mod2}

\subsection{Geometric setting}\label{agd}
Let $n$ be a positive integer. 
The variety $S_{2n}$
is isomorphic to the projective variety $S'_{2n}$
defined by $\sum_{i=1}^n(x_i^q y_i-x_i y_i^q)=0$
in $\mathbb{P}_{\mathbb{F}_q}^{2n-1}$. 
We set $Y'_{2n}=\mathbb{P}_{\mathbb{F}_q}^{2n-1} \setminus S'_{2n}$. Then we have $Y_{2n} \simeq 
Y'_{2n,\mathbb{F}_{q^2}}$. 
Let 
\[
J=\begin{pmatrix}
\mathbf{0}_n & E_n \\
-E_n & \mathbf{0}_n
\end{pmatrix}
 \in \GL_{2n}(\bF_q).
\] 
Let $\Sp_{2n}$ be the symplectic group over $\bF_q$ 
defined by the symplectic form 
\[
 \bF_q^n \times \bF_q^n 
 \to \bF_q;\ 
 (v,v') \mapsto 
 {}^t v J v'. 
\] 
Let 
$\Sp_{2n}(\bF_q)$ act on $\mathbb{P}_{\mathbb{F}_q}^{2n-1}$ by left multiplication. 
This action stabilizes $Y'_{2n}$. 
Let $\widetilde{Y}'_{2n}$
be the affine smooth variety defined by $\sum_{i=1}^n(x_i^q y_i-x_i y_i^q)=1$ in $\mathbb{A}_{\mathbb{F}_q}^{2n}$. This affine variety admits a similar 
action of $\Sp_{2n}(\bF_q)$. 
Similarly to \eqref{torsor1}, 
we have the $\Sp_{2n}(\bF_q)$-equivariant 
$\mu_{q+1}$-covering 
\begin{equation}\label{torsor2}
\widetilde{Y}'_{2n}
\to Y'_{2n};\ (x_1,\ldots,x_n,y_1,\ldots,y_n) 
\mapsto 
[x_1:\cdots:x_n:y_1:\cdots:y_n].
\end{equation}

Let $\mathrm{Fr}_q \in \mathrm{Gal}(\overline{\mathbb{F}}_q/\mathbb{F}_q)$ be the geometric Frobenius automorphism 
defined by $x \mapsto x^{q^{-1}}$ for $x \in \overline{\mathbb{F}}_q$. 
For a separated and of finite type scheme $Z$ over $\mathbb{F}_q$, let $\mathrm{Fr}_q$ denote 
the pull-back of $\mathrm{Fr}_q$ on $H_{\mathrm{c}}^i(Z_{\overline{\mathbb{F}}_q},\overline{\mathbb{Q}}_{\ell})$. 

Let $X'_{2n}$ be the affine smooth variety defined by 
\[
 z^q-z=
 \sum_{i=1}^n(x_i y_i^q - x_i^q y_i ) 
\] 
in $\mathbb{A}_{\mathbb{F}_q}^{2n+1}
=\Spec \mathbb{F}_q[x_1,\ldots,x_n,y_1,\ldots,y_n,z]$. 
We write $v=(x_1,\ldots,x_n,y_1,\ldots,y_n)$.  
Let $\oU_{2n}'$ be the unitary group over $\bF_q$ defined by 
the skew-hermitian form 
\[
 \mathbb{F}_{q^2}^n 
 \times \mathbb{F}_{q^2}^n
 \to \mathbb{F}_{q^2};\ 
 (v,v') \mapsto 
 {}^t \overline{v} J v'. 
\] 
The group $\oU_{2n}'(\bF_q)$ acts on 
$X'_{2n,\mathbb{F}_{q^2}}$ by 
$(v,z) \mapsto (g v,z)$ for $g \in \oU_{2n}'(\bF_q)$.  
Let $\mathbb{F}_q$ act on $X'_{2n}$ by 
$z \mapsto z+\eta$ for $\eta \in \mathbb{F}_q$. 

We put 
\[
 W_{n,\psi} = 
 H_{\mathrm{c}}^{2n}(X'_{2n,\overline{\mathbb{F}}_q},
 \overline{\mathbb{Q}}_{\ell}(n) )[\psi]. 
\]

We identify $\mu_{q+1}$ with the center of 
$\oU_{2n}'(\bF_q)$.  
By \cite[Lemma 3.4]{ITGeomHW}, the geometric Frobenius 
$\mathrm{Fr}_q$ stabilizes $W_{n,\psi}[\chi]$
for $\chi \in \mu_{q+1}^{\vee}$ such that $\chi^2=1$
and acts on it as an involution. 
Let $\kappa \in \{\pm\}$. 
For such $\chi$, let $W_{n,\psi}[\chi]^{\kappa}$ 
denote the $\kappa$-eigenspace of $\mathrm{Fr}_q$. 

Let $\nu$ be the quadratic character of $\mu_{q+1}$ if $p \neq 2$. 

\begin{lem}[{\cite[Lemma 7.1]{ITGeomHW}}]\label{lem:dim}
Let $n \geq 1$. 
We have 
\begin{align*}
\dim W_{n,\psi}[1]^{\kappa}=\frac{(q^n+\kappa)(q^n+\kappa q)}{2(q+1)}, \quad 
\dim W_{n,\psi}[\chi]=\frac{q^{2n}-1}{q+1}
\end{align*}  
for $\kappa \in \{\pm\}$ and $\chi \in \mu_{q+1}^{\vee} \setminus \{1\}$. 
Further, 
\[
 \dim W_{n,\psi}[\nu]^{\kappa} =\frac{q^{2n}-1}{2(q+1)}
\]
for $\kappa \in \{\pm\}$ if $p \neq 2$. 
\end{lem}

Let $\Lambda \in \{\overline{\mathbb{Q}}_{\ell}, \mathbb{F}\}$ 
and 
$\psi \in \Hom (\mathbb{F}_q,\Lambda^{\times}) \setminus \{1\}$. 
Let 
\begin{equation*}
\pi' \colon 
\mathbb{A}_{\mathbb{F}_q}^{2n} \to \mathbb{A}_{\mathbb{F}_q}^1;\ 
((x_i)_{1 \leq i \leq n},(y_i)_{1 \leq i \leq n}) 
\mapsto 
\sum_{i=1}^n (x_iy_i^q-x_i^qy_i). 
\end{equation*}
Then we have a natural isomorphism 
\begin{equation}\label{eq:X'Lpsi}
 H_{\mathrm{c}}^{2n}(X'_{2n,\overline{\mathbb{F}}_q},
 \Lambda )[\psi] \simeq 
 H_{\mathrm{c}}^{2n}(\mathbb{A}^{2n},\pi'^\ast \mathscr{L}_{\psi}). 
\end{equation}
Let $\mathbf{0} \in \mathbb{A}_{\mathbb{F}_q}^{2n}$
be the origin.
Let $U'=\pi'^{-1}(\bG_{\mathrm{m},\mathbb{F}_q})$, 
$Z'=\pi'^{-1}(0)$ and $Z'^0=Z' \setminus \{\mathbf{0}\}$.  
In the following, for a $\mu_{q+1}$-representation $M$ over $\Lambda$, 
let $M[1]$ denote the $\mu_{q+1}$-fixed part of $M$. 

\begin{lem}\label{tl}
Assume that $q+1$ is invertible in $\Lambda$ and $n \geq 2$. 
Then we have 
\[
H_{\mathrm{c}}^{2n}(Y'_{2n,\overline{\mathbb{F}}_q},\Lambda)=0, \quad H_{\mathrm{c}}^{2n+1}(U'_{\overline{\mathbb{F}}_q},\pi'^\ast \mathscr{L}_{\psi})[1]=0. 
\] 
\end{lem}
\begin{proof}
The first claim follows from \eqref{T} using 
the isomorphism 
$Y_{2n,\overline{\mathbb{F}}_q} \simeq Y'_{2n,\overline{\mathbb{F}}_q}$ and 
the assumption 
that $q+1$ is invertible in $\Lambda$. 
The second claim follows from the first one and 
\cite[Lemma 4.3, Remark 7.11]{ITGeomHW}. 
\end{proof}

\begin{lem}\label{tl2}
Assume that $q+1$ is invertible in $\Lambda$. 
We have an isomorphism
\[
H_{\mathrm{c}}^{2n}(\mathbb{A}^{2n},\pi'^\ast \mathscr{L}_{\psi})[1] \simeq 
H_{\mathrm{c}}^{2n-1}(Y'_{2n,\overline{\mathbb{F}}_q},\Lambda (-1)) 
\] 
as representations of 
$\oU_{2n}'(\bF_q)$ and $\Gal (\ol{\bF}_q/\bF_q)$. 
\end{lem}
\begin{proof}
Since $\pi'^\ast \mathscr{L}_{\psi}|_{Z'}=\Lambda$, 
we have an exact sequence 
\begin{equation*}
 H_{\mathrm{c}}^{2n}(\mathbb{A}^{2n},\pi'^\ast \mathscr{L}_{\psi})[1] 
 \longrightarrow 
 H_{\mathrm{c}}^{2n}(Z'_{\overline{\mathbb{F}}_q},\Lambda)[1] \to 0 
\end{equation*} 
by Lemma \ref{tl} and the assumption that $q+1$ is invertible in $\Lambda$. 
By $n \geq 1$ and \cite[Lemma 4.4 (3), Remark 7.11]{ITGeomHW}, we have 
\[
 H_{\mathrm{c}}^{2n}(Z'_{\overline{\mathbb{F}}_q},\Lambda)[1] \simeq 
 H_{\mathrm{c}}^{2n}(Z'^0_{\overline{\mathbb{F}}_q},\Lambda)[1] 
 = H_{\mathrm{c}}^{2n}(Z'^0_{\overline{\mathbb{F}}_q},\Lambda). 
\]
We have a morphism 
\[
 H_{\mathrm{c}}^{2n}(Z'^0_{\overline{\mathbb{F}}_q},\Lambda) \longrightarrow 
 H^{2n-2}(S'_{2n,\overline{\mathbb{F}}_q},\Lambda(-1) ) 
\]
by \cite[Lemma 4.4, Remark 7.11]{ITGeomHW}, whose cokernel is a sum of trivial representations. 
Further we have a surjective morphism 
\[
 H^{2n-2}(S'_{2n,\overline{\mathbb{F}}_q},\Lambda(-1) ) \longrightarrow 
 H^{2n-1}(Y'_{2n,\overline{\mathbb{F}}_q},\Lambda(-1) ) 
\]
by the long exact sequence for 
$Y'_{2n}=\mathbb{P}_{\mathbb{F}_q}^{2n-1} \setminus S'_{2n}$. 
Consider the composition of the above morphisms 
\begin{equation}\label{eq:comp}
 H_{\mathrm{c}}^{2n}(\mathbb{A}^{2n},\pi'^\ast \mathscr{L}_{\psi})[1] 
 \longrightarrow 
 H^{2n-1}(Y'_{2n,\overline{\mathbb{F}}_q},\Lambda(-1) ) . 
\end{equation}
We have 
\begin{equation}\label{eq:AYdim}
 \dim H_{\mathrm{c}}^{2n}(\mathbb{A}^{2n},\pi'^\ast \mathscr{L}_{\psi})[1]
 = \dim H^{2n-1}(Y'_{2n,\overline{\mathbb{F}}_q},\Lambda(-1) ) 
 =\frac{q^{2n}+q}{q+1}. 
\end{equation} 
by \cite[(2.6), Proposition 2.6, Lemma 4.2]{ITGeomHW}, 
Proposition \ref{prop:Vn} \ref{en:Vdim}, 
\eqref{T2} and \eqref{T2-5}. 
The $\oU_{2n}(\bF_q)$-representation 
$H^{2n-1}(Y'_{2n,\overline{\mathbb{F}}_q},\Lambda(-1) )$ 
is irreducible of dimension greater than $1$ 
by Proposition \ref{ccc} \ref{en:irr1} and \eqref{eq:AYdim}. 
Hence \eqref{eq:comp} is surjective, 
since the cokernel of \eqref{eq:comp} is 
a sum of trivial representations. 
Therefore 
\eqref{eq:comp} is an isomorphism 
by \eqref{eq:AYdim}. 
\end{proof}

\subsection{Invariant part}\label{ssec:Inv}
In this subsection, we study some invariant parts of 
$H_{\mathrm{c}}^{2n-1}(Y'_{2n,\overline{\mathbb{F}}_q},\mathbb{F})$. 

Let $U$ be the unipotent radical of the Borel subgroup of 
$\SL_2$ consisting of upper triangular matrices. 
Recall that we have the isomorphisms 
\begin{gather}\label{dickson}
\begin{aligned}
& \mathbb{A}_{\mathbb{F}_q}^2/U(\bF_q) 
\xrightarrow{\sim} \mathbb{A}_{\mathbb{F}_q}^2;\ 
(x,y) \mapsto (x^q-xy^{q-1},y), \\
& \mathbb{A}_{\mathbb{F}_q}^2/\SL_2(\bF_q) \xrightarrow{\sim} \mathbb{A}_{\mathbb{F}_q}^2;\ (x,y) \mapsto \left(x^qy-xy^q,\frac{x^{q^2}y-xy^{q^2}}{x^qy-xy^q}\right) 
\end{aligned}
\end{gather}
(\cf \cite[Exercise 2.2 (b), (e)]{BonRepSL2}).

We regard a product group 
$\SL_2(\bF_q)^n$ as a subgroup of 
$\Sp_{2n}(\bF_q)$ by the injective homomorphism 
\begin{equation*}
\SL_2(\bF_q)^n \hookrightarrow \Sp_{2n}(\bF_q);\ 
\left(\begin{pmatrix}
a_i & b_i \\
c_i & d_i 
\end{pmatrix}
\right)_{1 \leq i \leq n} \mapsto 
\begin{pmatrix}
\mathrm{diag}(a_1,\ldots,a_n) & \mathrm{diag}(b_1,\ldots,b_n) \\
\mathrm{diag}(c_1,\ldots,c_n) & \mathrm{diag}(d_1,\ldots,d_n)
\end{pmatrix}. 
\end{equation*}
We understand the quotients 
$\widetilde{Y}'_{2n}/U(\bF_q)^n$ and 
$\widetilde{Y}'_{2n}/\SL_2(\bF_q)^n$, respectively. 
By \eqref{dickson}, we have the isomorphisms 
\begin{gather}\label{dickson2}
\begin{aligned}
& \widetilde{Y}'_{2n}/U(\bF_q)^n
\xrightarrow{\sim} 
\left\{((s_i)_{1 \leq i \leq n},(t_i)_{1 \leq i \leq n}) \in \mathbb{A}_{\mathbb{F}_{q^2}}^{2n}\ \Big|\  \sum_{i=1}^n s_i t_i=1\right\}; \\ 
& ((x_i)_{1 \leq i \leq n},(y_i)_{1 \leq i \leq n}) \mapsto 
\left((x_i^q-x_iy_i^{q-1})_{1 \leq i \leq n}, (y_i)_{1 \leq i \leq n}\right), \\
& \widetilde{Y}'_{2n}/\SL_2(\bF_q)^n
\xrightarrow{\sim} 
\left\{((s_i)_{1 \leq i \leq n},(t_i)_{1 \leq i \leq n}) \in \mathbb{A}_{\mathbb{F}_{q^2}}^{2n}\ \Big|\  \sum_{i=1}^n s_i=1\right\} \simeq 
\mathbb{A}_{\mathbb{F}_{q^2}}^{2n-1};\\  
& ((x_i)_{1 \leq i \leq n},(y_i)_{1 \leq i \leq n}) \mapsto 
\left((x_i^q y_i-x_i y_i^q)_{1 \leq i \leq n},
\left(\frac{x_i^{q^2} y_i-x_i y_i^{q^2}}{x_i^qy_i-x_iy_i^q}\right)_{1 \leq i \leq n}\right). 
\end{aligned}
\end{gather}
The actions of $U(\bF_q)^n$ and $\SL_2(\bF_q)^n$ on $\widetilde{Y}'_{2n}$ are not free if $n \geq 2$. 

The following proposition plays a key role to show 
Proposition \ref{sum} and corresponding results in 
the case where $p \neq 2$.
\begin{prop}\label{sol}
\begin{enumerate}
\item\label{en:invtY} 
We have 
\[
H_{\mathrm{c}}^{2n-1}(\widetilde{Y}'_{2n,\overline{\mathbb{F}}_q},\mathbb{F})^{\SL_2(\bF_q)^n}=0, \quad H_{\mathrm{c}}^{2n-1}(\widetilde{Y}'_{2n,\overline{\mathbb{F}}_q},\mathbb{F})^{U(\bF_q)^n} \simeq \mathbb{F}.
\] 
\item\label{en:invY} 
We have 
\[
 H_{\mathrm{c}}^{2n-1}(Y'_{2n,\overline{\mathbb{F}}_q},\mathbb{F})^{\Sp_{2n}(\bF_q)}=0,\quad   
 H_{\mathrm{c}}^{2n-1}(Y'_{2n,\overline{\mathbb{F}}_q},\bF )^{U(\bF_q)^n}
\simeq \bF . 
\]
\end{enumerate}
\end{prop}
\begin{proof}
We show the claim \ref{en:invtY} 
by induction on $n$. 
The action of $\SL_2(\bF_q)$ on $\widetilde{Y}'_2$ is 
free by \cite[Proposition 2.1.2]{BonRepSL2}. Hence, the claim 
for $n=1$ follows from 
Lemma \ref{affine2} \ref{en:Gmid}, \eqref{dickson2} and 
$H_{\mathrm{c}}^1(\bG_{\mathrm{m}},\bF) \simeq \bF$.

Assume $n \geq 2$. 
We consider the closed subscheme $R_{2n}$ of $\widetilde{Y}'_{2n,\overline{\mathbb{F}}_q}$ defined by $y_n=0$. 
This is isomorphic to $\mathbb{A}^1 \times \widetilde{Y}'_{2n-2,\overline{\mathbb{F}}_q}$. 
Let $Q_{2n}=\widetilde{Y}'_{2n,\overline{\mathbb{F}}_q} \setminus 
R_{2n}$. 
Similarly to \eqref{dickson2}, 
the quotient $Q_{2n}/U(\bF_q)$ is isomorphic to 
$\mathbb{A}^{2n-2} \times \bG_{\mathrm{m}}$. 
Therefore,  
$H_{\mathrm{c}}^{i}(Q_{2n},\mathbb{F})^{U(\bF_q)}$ is 
zero for $i=2n-1, 2n$ by $n \geq 2$.
Hence, we have isomorphisms
\[
H_{\mathrm{c}}^{2n-1}(\widetilde{Y}'_{2n,\overline{\mathbb{F}}_q},\mathbb{F})^{U(\bF_q)}
\xrightarrow{\sim} H_{\mathrm{c}}^{2n-1}
(R_{2n},\mathbb{F})^{U(\bF_q)} \simeq 
H_{\mathrm{c}}^{2n-3}(\widetilde{Y}'_{2n-2,\overline{\mathbb{F}}_q},\mathbb{F}), 
\]
which are compatible with the actions of $\SL_2(\bF_q)^{n-1}$. 
Hence, the claim follows from the induction hypothesis. 

We show the claim \ref{en:invY}. 
By applying Lemma \ref{affine2} \ref{en:Gmid} 
to the $\mu_{q+1}$-torsor \eqref{torsor2}, we have 
\begin{equation}\label{nzero}
H_{\mathrm{c}}^{2n-1}(Y'_{2n,\overline{\mathbb{F}}_q},\mathbb{F}) \simeq 
H_{\mathrm{c}}^{2n-1}(\widetilde{Y}'_{2n,\overline{\mathbb{F}}_q},\mathbb{F})^{\mu_{q+1}}  \quad \textrm{for any $n \geq 1$}. 
\end{equation}
We have 
$H_{\mathrm{c}}^{2n-1}(\widetilde{Y}'_{2n,\overline{\mathbb{F}}_q},\mathbb{F})^{\Sp_{2n}(\bF_q)}=0$. By \eqref{nzero}, we have the inclusion 
$H_{\mathrm{c}}^{2n-1}(Y'_{2n,\overline{\mathbb{F}}_q},\mathbb{F}) \subset H_{\mathrm{c}}^{2n-1}(\widetilde{Y}'_{2n,\overline{\mathbb{F}}_q},\mathbb{F})$. 
Hence, we have 
$H_{\mathrm{c}}^{2n-1}(Y'_{2n,\overline{\mathbb{F}}_q},\mathbb{F})^{\Sp_{2n}(\bF_q)}=0$. 

The action of $\mu_{q+1}$ on $\widetilde{Y}'_{2n}$ commutes with the one of $U(\bF_q)^n$. 
By the above proof, we have an isomorphism 
\[
H_{\mathrm{c}}^{2n-1}(\widetilde{Y}'_{2n,\overline{\mathbb{F}}_q},\mathbb{F})^{U(\bF_q)^n} \simeq 
H_{\mathrm{c}}^1(\bG_{\mathrm{m}},\mathbb{F})
\]
as $\mathbb{F}[\mu_{q+1}]$-modules, where 
$\mu_{q+1}$ acts on $\bG_{\mathrm{m}}$ by the usual 
multiplication. 
Hence $\mu_{q+1}$ acts on $H_{\mathrm{c}}^{2n-1}(\widetilde{Y}'_{2n,\overline{\mathbb{F}}_q},\mathbb{F})^{U(\bF_q)^n}$ trivially. 
Therefore, we have 
\begin{equation*}
H_{\mathrm{c}}^{2n-1}(Y'_{2n,\overline{\mathbb{F}}_q},\mathbb{F})^{U(\bF_q)^n} \simeq 
H_{\mathrm{c}}^{2n-1}(\widetilde{Y}'_{2n,\overline{\mathbb{F}}_q},\mathbb{F})^{U(\bF_q)^n \times \mu_{q+1}} \simeq \mathbb{F} 
\end{equation*}
by \eqref{nzero}. 
\end{proof}

\subsection{Trace computations}\label{try0}
In this subsection, we assume $p \neq 2$. 
An aim in this subsection is to show Proposition \ref{433}, 
which implies that the Brauer characters associated to  
$W_{n,\psi}[\nu]^+$ and $W_{n,\psi}[\nu]^-$ are distinct in the case where $\ell \neq 2$ (\cf Proposition \ref{prop:nc1}). 

Let $\bigl(\frac{a}{\mathbb{F}_q}\bigr)=a^{\frac{q-1}{2}}$ 
for $a \in \mathbb{F}_q^{\times}$. 
For $\psi \in \Hom (\bF_q,\Lambda^{\times}) \setminus \{ 1 \}$, 
we consider the quadratic Gauss sum  
\[
 G(\psi)
 =\sum_{x \in \mathbb{F}_q^{\times}} 
 \left(\frac{x}{\mathbb{F}_q}\right) \psi(x) \in \Lambda. 
\]
As a well-known fact, we have 
$G(\psi)^2=\bigl(\frac{-1}{\mathbb{F}_q}\bigr)q$. 
In particular, we have $G(\psi) \neq 0$.

Let $X$ be the affine smooth surface defined by 
$z^q-z=x y^q-x^qy$ in $\mathbb{A}_{\mathbb{F}_q}^3$. 
We consider the projective smooth surface  
$\overline{X}$ defined by 
\[
Z_2^q Z_3-Z_2 Z_3^q=Z_0 Z_1^q-Z_0^q Z_1
\]
in $\mathbb{P}_{\mathbb{F}_q}^3
=\Proj \mathbb{F}_q[Z_0,Z_1,Z_2,Z_3]$. 
We regard $X$ as an open subscheme of 
$\overline{X}$
by $(x,y,z) \mapsto [x:y:z:1]$. 
Let $D=\overline{X} \setminus X$. 
Let 
\begin{equation}\label{uu}
u=\begin{pmatrix}
1 & 1 \\
0 & 1
\end{pmatrix} \in \SL_2(\bF_q), 
\end{equation}
which is of order $p$. 
Let $F$ denote the Frobenius endomorphism of $X$
over $\mathbb{F}_q$. 
Let $\eta \in \mathbb{F}_q$ and $\zeta \in \mu_{q+1}$. 
Let $f_{\eta,\zeta}$ denote the endomorphism 
$F \eta \zeta u$ of $X_{\overline{\mathbb{F}}_q}$. 
This endomorphism extends to the one of 
$\overline{X}_{\overline{\mathbb{F}}_q}$ given by 
\begin{align*}
f_{\eta,\zeta}  \colon \overline{X}_{\overline{\mathbb{F}}_q} \to \overline{X}_{\overline{\mathbb{F}}_q};\ 
 [Z_0:Z_1:Z_2:Z_3] \mapsto 
\left[(Z_0+Z_1)^q: Z_1^q: \zeta (Z_2+\eta Z_3)^q: \zeta Z_3^q\right]. 
\end{align*}
This endomorphism $f_{\eta,\zeta}$ stabilizes $D_{\overline{\mathbb{F}}_q}$. 

\begin{lem}\label{try3}
We have 
\[
\Tr (f_{\eta,\zeta};H^\ast (\overline{X}_{\overline{\mathbb{F}}_q},\overline{\mathbb{Q}}_{\ell}))=
\begin{cases}
q^2+q+1 & \textrm{if $\eta=0$}, \\
2q^2+q+1 & \textrm{if $\eta \neq 0$ and $\nu(\zeta) \bigl(\frac{-\eta}{\mathbb{F}_q}\bigr)=1$}, \\
q+1 & \textrm{if $\eta \neq 0$ and $\nu(\zeta) \bigl(\frac{-\eta}{\mathbb{F}_q}\bigr)=-1$}. 
\end{cases}
\]
\end{lem}
\begin{proof}
By the Grothendieck--Lefschetz trace formula, 
it suffices to count the number of fixed points of $f_{\eta,\zeta}$
on $\overline{X}_{\overline{\mathbb{F}}_q}$
with multiplicity. 
The set of fixed points of $f_{\eta,\zeta}$ equals 
the union of the two sets 
\begin{align*}
\Sigma_1&= \{[x:y:z:1] \in \mathbb{P}^3 \mid x^q-\zeta x=-\zeta y ,\ 
 y^q=\zeta y,\ y^{q+1}=-\eta,\ 
z^q-z=-\eta \}, \\ 
\Sigma_2&=\{[0:z:1:0] \in \mathbb{P}^3 \mid z^q=\zeta z\} \cup
\{[0:1:0:0]\}.   
\end{align*}
Assume that $\eta \neq 0$ and $\nu(\zeta)\bigl(\frac{-\eta}{\mathbb{F}_q}\bigr)=1$. 
We have 
\[
 \Sigma_1=\{[x:y:z:1] \in \mathbb{P}^3 \mid x^q-\zeta x=-\zeta y,\ 
 y^2=-\eta/\zeta,\ 
 z^q-z=-\eta \}
\]
and $\lvert \Sigma_1 \rvert=2q^2$. 
One can check that the multiplicity of $f_{\eta,\zeta}$
at any point of $\Sigma_1 \cup \Sigma_2$ 
equals one.  
Hence the claim in this case follows. 
Assume that $\eta \neq 0$ and $\nu(\zeta)\bigl(\frac{-\eta}{\mathbb{F}_q}\bigr)=-1$. Then we have 
$\lvert \Sigma_1 \rvert=0$. 
The claim is shown in the same way as above. 
The other case is computed similarly. 
\end{proof}
We simply write $f_{\zeta}$ for $f_{0,\zeta}$.  
Let $\psi \in \mathbb{F}_q^{\vee} \setminus \{1\}$. 
we simply write $\mathscr{L}^0_{\psi}$ for 
the pullback of $\mathscr{L}_{\psi}$ under 
$\bA^2 \to \bA^1 ; \ (x,y) \mapsto xy^q-x^q y$. 

\begin{cor}\label{try3-5}
We have 
\[
\frac{1}{q+1}
\sum_{\zeta \in \mu_{q+1}} \nu(\zeta) \Tr(f_{\zeta}; 
H_{\mathrm{c}}^2(\mathbb{A}^2,\mathscr{L}^0_{\psi}(1)))=G(\psi).
\]
\end{cor}
\begin{proof}
For any $i$, we have isomorphisms 
\[
H_{\mathrm{c}}^i(\mathbb{A}^2,\mathscr{L}^0_{\psi}(1))
\simeq H_{\mathrm{c}}^i(X_{\overline{\mathbb{F}}_q},\overline{\mathbb{Q}}_{\ell}(1))
[\psi]
\xrightarrow{\sim} 
H_{\mathrm{c}}^i(\overline{X}_{\overline{\mathbb{F}}_q},\overline{\mathbb{Q}}_{\ell}(1))[\psi], 
\] 
where the second isomorphism follows, 
since the group $\mathbb{F}_q$ acts trivially on $D$. 
By the K\"{u}nneth formula and \cite[Lemma 3.3]{ITGeomHW}, 
we have 
\begin{align*}
\sum_{\zeta \in \mu_{q+1}} \nu(\zeta) 
\Tr(f_{\zeta}; H_{\mathrm{c}}^2(\mathbb{A}^2,\mathscr{L}^0_{\psi}(1))) 
&=
\sum_{\zeta \in \mu_{q+1}} \nu(\zeta) \Tr(f_{\zeta}; 
H_{\mathrm{c}}^\ast(\mathbb{A}^2,\mathscr{L}^0_{\psi}(1))) \\
&=
\sum_{\zeta \in \mu_{q+1}} \nu(\zeta) \Tr(f_{\zeta}; 
H^\ast(\overline{X}_{\overline{\mathbb{F}}_q},\overline{\mathbb{Q}}_{\ell}(1))
[\psi])\\
&=\frac{1}{q}\sum_{\zeta \in \mu_{q+1}}
\nu(\zeta) \sum_{\eta \in \mathbb{F}_q} \psi^{-1}(\eta)
\Tr(f_{\eta,\zeta}; H^\ast (\overline{X}_{\overline{\mathbb{F}}_q},\overline{\mathbb{Q}}_{\ell}(1))). 
\end{align*}
The last term equals 
\begin{align*}
& \frac{1}{q^2}\sum_{\zeta \in \mu_{\frac{q+1}{2}}} \left((2q^2+q+1) 
\sum_{\eta \in (\mathbb{F}_q^{\times})^2} \psi (\eta)
+(q+1)\sum_{\eta \notin (\mathbb{F}_q^{\times})^2} \psi (\eta)\right) \\
& -\frac{1}{q^2}
\sum_{\zeta \notin \mu_{\frac{q+1}{2}}} \left((2q^2+q+1) 
\sum_{\eta \notin (\mathbb{F}_q^{\times})^2} \psi (\eta)
+(q+1)\sum_{\eta \in (\mathbb{F}_q^{\times})^2} \psi (\eta)\right)=(q+1) G(\psi) 
\end{align*}
by Lemma \ref{try3} and $\sum_{\zeta \in \mu_{q+1}}\nu(\zeta)=0$. 
\end{proof}
\begin{lem}\label{try4}
We have 
\[
\Tr(F\eta \zeta; H^{\ast}(\overline{X}_{\overline{\mathbb{F}}_q},\overline{\mathbb{Q}}_{\ell}))
=\begin{cases}
(q+1)(q^2+1) & \textrm{if $\eta=0$}, \\
q^2+q+1 & \textrm{if $\eta \neq 0$}. 
\end{cases} 
\]
\end{lem}
\begin{proof}
The set of fixed points of $F\eta \zeta$ on $\overline{X}_{\overline{\mathbb{F}}_q}$ is 
the union of the three finite sets
\begin{align*}
 \Sigma_1&=\{[x:y:z:1]  \in \mathbb{P}^3\mid z^q-z=xy^q-x^qy=-\eta,\ x^q=\zeta x,\ y^q=\zeta y\}, \\
 \Sigma_2&=\{[x:y:1:0] \in \mathbb{P}^3 \mid x^q=\zeta x,\ y^q=\zeta y\}, \\  
 \Sigma_3&=\{[Z_0:Z_1:0:0]  \in \mathbb{P}^3\mid [Z_0:Z_1] \in \mathbb{P}^1_{\bF_q}(\mathbb{F}_q)\}. 
\end{align*}
We have $\Sigma_1=\emptyset$ if $\eta \neq 0$ and 
$\lvert \Sigma_1 \rvert =q^3$ if $\eta=0$. 
The multiplicity of $F\eta \zeta$ at any point of 
$\bigcup_{i=1}^3 \Sigma_i$ equals one. Hence the claim follows. 
\end{proof}
\begin{cor}\label{try5}
Let $\psi \in \mathbb{F}_q^{\vee} \setminus \{1\}$. 
We have 
$\Tr(F\zeta; H_{\mathrm{c}}^2(\mathbb{A}^2,\mathscr{L}^0_{\psi}(1)))=q$. 
\end{cor}
\begin{proof}
Similarly to the proof of Corollary \ref{try3-5}, we have 
\begin{align*}
\Tr(F\zeta; H_{\mathrm{c}}^2(\mathbb{A}^2,\mathscr{L}^0_{\psi}(1)))
&=\Tr(F\zeta; H_{\mathrm{c}}^\ast(\mathbb{A}^2,\mathscr{L}^0_{\psi}(1))) \\
&=\frac{1}{q^2} \sum_{\eta \in \mathbb{F}_q} \psi^{-1}(\eta)
\Tr(F\eta \zeta; H^{\ast}(\overline{X}_{\overline{\mathbb{F}}_q},\overline{\mathbb{Q}}_{\ell}(1)))=q 
\end{align*}
by Lemma \ref{try4}. 
\end{proof}

\begin{prop}\label{433}
Let $g_0=(u,1,\ldots,1) \in \SL_2(\bF_q)^n \subset 
\Sp_{2n}(\bF_q)$, where 
$u \in \SL_2(\bF_q)$ is as in \eqref{uu}. 
We have 
\[
W_{n,\psi}[\nu]^+(g_0)-W_{n,\psi}[\nu]^-(g_0)
=q^{n-1} G(\psi). 
\]
In particular, we have $W_{n,\psi}[\nu]^+(g_0) \neq W_{n,\psi}[\nu]^-(g_0)$. 
\end{prop}
\begin{proof}
Let $\kappa \in \{\pm\}$. 
Let $\nu_{\kappa}$ be the character of 
$\mu_{q+1} \rtimes \Gal (\bF_{q^2}/\bF_q)$ 
extending $\nu$ by the condition 
$\nu_{\kappa}(\mathrm{Fr}_q)=\kappa$. 
For $g \in \Sp_{2n}(\bF_q)$, 
the trace $W_{n,\psi}[\nu]^{\kappa}(g)=W_{n,\psi}[\nu_{\kappa}](g)$ equals 
\begin{align*}
&\frac{1}{\lvert \mu_{q+1} \rtimes \Gal (\bF_{q^2}/\bF_q) \rvert} 
 \sum_{h \in \mu_{q+1} \rtimes \Gal (\bF_{q^2}/\bF_q)} 
 \nu_{\kappa}(h)^{-1} W_{n,\psi} (h g) \\ 
&=\frac{1}{2(q+1)} \sum_{\zeta \in \mu_{q+1}} \nu(\zeta)
\left(\Tr(\zeta g; H_{\mathrm{c}}^{2n}(\mathbb{A}^{2n},\pi'^\ast \mathscr{L}_{\psi}(n)))+\kappa
\Tr(F \zeta g; H_{\mathrm{c}}^{2n}(\mathbb{A}^{2n},\pi'^\ast \mathscr{L}_{\psi}(n)))\right). 
\end{align*} 
Hence we have 
\begin{align*}
& W_{n,\psi}[\nu]^+(g_0)
-W_{n,\psi}[\nu]^-(g_0) \\
&=
 \frac{1}{q+1} \sum_{\zeta \in \mu_{q+1}}
\nu(\zeta) \Tr (F \zeta g_0; H_{\mathrm{c}}^{2n}(\mathbb{A}^{2n},\pi'^\ast \mathscr{L}_{\psi}(n)))\\
& =
\frac{1}{q+1} \sum_{\zeta \in \mu_{q+1}}
\nu(\zeta) \Tr(F\zeta u; H_{\mathrm{c}}^2(\mathbb{A}^2,\mathscr{L}^0_{\psi}(1))) \Tr(F\zeta; H_{\mathrm{c}}^2(\mathbb{A}^2,\mathscr{L}^0_{\psi}(1)))^{n-1} \\
&= \frac{q^{n-1}}{q+1} \sum_{\zeta \in \mu_{q+1}}
\nu(\zeta) \Tr(F\zeta u; H_{\mathrm{c}}^2(\mathbb{A}^2,\mathscr{L}^0_{\psi}(1)))
=q^{n-1} G(\psi),  
\end{align*}
where the second equality 
follows from the K\"{u}nneth formula, the third one 
follows from Corollary \ref{try5} and 
the last one follows from 
Corollary \ref{try3-5}. 
\end{proof}

\section{Representation of $\Sp_{2n}(\bF_q)$}\label{sec:RepSP}

\subsection{Weil representation of $\Sp_{2n}(\bF_q)$}
Assume that $p \neq 2$ in this subsection. 
Let $\psi \in \mathbb{F}_q^{\vee} \setminus \{1\}$. 
A Weil representation of $\Sp_{2n}(\bF_q)$
associated to $\psi$ 
is studied in \cite{GerWeil} and \cite{HoweCharW}, 
which we denote by $\omega_{\psi}$. 
This has dimension $q^n$, and 
splits to two irreducible 
representations $\omega_{\psi,+}$ and $\omega_{\psi,-}$, which are of dimensions $(q^n+1)/2$
and $(q^n-1)/2$, respectively by \cite[Corollary 4.4 (a)]{GerWeil}. 
For $\psi \in \mathbb{F}_q^{\vee}$ and 
$a \in \mathbb{F}_q$, let $\psi_a$ denote the character of 
$\mathbb{F}_q$ defined by 
$x \mapsto \psi(ax)$ for $x \in \mathbb{F}_q$.  
For $\psi \in \mathbb{F}_q^{\vee} \setminus \{1\}$, it is known that $\omega_{\psi,\kappa} \simeq 
\omega_{\psi_{a},\kappa}$ if and only if 
$a \in (\mathbb{F}_q^{\times})^2$ by \cite[Corollary 2.12]{ShiChWe}. 

For an element $s \in \SO_{2n+1}(\bF_q)$, 
let $\Spec (s)$ denote the set of 
the eigenvalues of $s$ as an element of 
$\GL_{2n+1}(\bF_q)$. 
For $\kappa \in \{\pm\}$, 
let $s_{\kappa} \in \SO_{2n+1}(\bF_q)$ 
be a semisimple element such that $\Spec(s_{\kappa})=\{1,-1,\ldots,-1\}$ 
and $C_{\SO_{2n+1}(\bF_q)}(s_{\kappa})=\oO_{2n}^{\kappa}(\bF_q)$. 

\begin{lem}\label{lem:omels}
We have 
$\omega_{\psi,\kappa} \in \cE(\Sp_{2n},(s_{\kappa}))$ 
for $\kappa \in \{\pm\}$. 
\end{lem}
\begin{proof}
We know that there are two irreducible representations 
in $\cE(\Sp_{2n},(s_{\kappa}))$ of degree 
$(q^n + \kappa)/2$ by 
\cite[13.23 Theorem, 13.24 Remark]{DiMiRepLie}. 
Hence the claim follows from 
\cite[Lemma 4.9]{LOSTOre}. 
\end{proof}

By Lemma \ref{lem:lirr}, 
these $\omega_{\psi,\kappa}$ 
remain irreducible after mod $\ell$ reduction.  
These mod $\ell$ irreducible modules of 
$\Sp_{2n}(\bF_q)$ are called 
\textit{Weil modules} in \cite[\S 5]{GMSTCross}. 
We will use this terminology later. 
There are just two Weil modules for each dimension 
(\cf \cite[p.~305]{GMSTCross}).

\begin{lem}\label{weil}
Assume that $p \neq 2$. 
We have $\omega_{\psi,\kappa} \notin 
\mathcal{E}_{\ell}
(\Sp_{2n},(1))$ for 
any $\psi \in \mathbb{F}_q^{\vee} \setminus \{1\}$
and $\kappa \in \{\pm\}$.  
\end{lem}
\begin{proof}
This follows from Lemma \ref{lem:diflb}, 
Lemma \ref{lem:omels} and $\ell \neq 2$. 
\end{proof}
\begin{rem}
If $n=2$, $p \neq 2$ and $\ell=2$, 
the representation $\omega_{\psi,\kappa}$ 
belongs to the principal block 
by \cite[p.~710]{WhiDecndiv}. 
\end{rem}

\subsection{Frobenius action}
In the sequel, every cohomology 
is an $\Sp_{2n}(\bF_q)$-representation, 
and every homomorphism is 
$\Sp_{2n}(\bF_q)$-equivariant. 
Let 
$\Lambda \in \{ \overline{\mathbb{Q}}_{\ell}, 
 \mathcal{O}, \ol{\bF}_{\ell} \}$. 
Let $\psi \in \Hom (\bF_q,\Lambda^{\times}) \setminus \{ 1 \}$ 
and $\chi \in \Hom (\mu_{q+1},\Lambda^{\times})$ 
such that $\chi^2=1$. 
We set 
\[
 H_{\mathrm{c}}^{2n-1}(Y'_{2n,\overline{\mathbb{F}}_q},\sK_{\chi})^{\kappa} = 
 \begin{cases}
 H_{\mathrm{c}}^{2n-1}(Y'_{2n,\overline{\mathbb{F}}_q},\Lambda)^{\mathrm{Fr}_q=\kappa q^{n-1}} & \textrm{if $\chi =1$,}\\ 
 H_{\mathrm{c}}^{2n-1}(Y'_{2n,\overline{\mathbb{F}}_q},\sK_{\nu})^{\mathrm{Fr}_q=-\kappa q^n G(\psi)^{-1}} & \textrm{if $\chi =\nu$}
\end{cases}
\]
for $\kappa \in \{\pm\}$. 
\begin{lem}
\begin{enumerate}
\item\label{en:dec}
We have a decomposition 
\begin{equation*}
H_{\mathrm{c}}^{2n-1}(Y'_{2n,\overline{\mathbb{F}}_q},\sK_{\chi}) 
\simeq 
H_{\mathrm{c}}^{2n-1}(Y'_{2n,\overline{\mathbb{F}}_q},\sK_{\chi})^+ \oplus 
H_{\mathrm{c}}^{2n-1}(Y'_{2n,\overline{\mathbb{F}}_q},\sK_{\chi})^- . 
\end{equation*}
\item
If $\Lambda=\cO$, 
we have isomorphisms 
\begin{gather}\label{irr-2}
\begin{aligned}
& H_{\mathrm{c}}^{2n-1}(Y'_{2n,\overline{\mathbb{F}}_q},\sK_{\chi})^{\kappa} \otimes_{\mathcal{O}}
\ol{\bF}_{\ell} \simeq H_{\mathrm{c}}^{2n-1}(Y'_{2n,\overline{\mathbb{F}}_q},\sK_{\ol{\chi}})^{\kappa}, \\
& 
H_{\mathrm{c}}^{2n-1}(Y'_{2n,\overline{\mathbb{F}}_q},\sK_{\chi})^{\kappa} \otimes_{\mathcal{O}} \overline{\mathbb{Q}}_{\ell} \simeq W_{n,\psi}[\chi]^{\kappa}
\end{aligned}
\end{gather}
for $\kappa \in \{\pm\}$. 
\end{enumerate}
\end{lem}
\begin{proof}
The claim \ref{en:dec} for $\Lambda=\ol{\bQ}_{\ell}$ and 
the second isomorphism in \eqref{irr-2} 
follow from 
Lemma \ref{tl2} and 
\cite[Lemma 3.4, Lemma 4.3, Corollary 4.6]{ITGeomHW}. 
Then the claim \ref{en:dec} for $\Lambda=\cO$ 
follows from Lemma \ref{affine} \ref{en:frmid}. 
Further, the claim \ref{en:dec} for $\Lambda=\ol{\bF}_{\ell}$ 
and the first isomorphism in \eqref{irr-2} follow from 
Proposition \ref{lp}. 
\end{proof}

\subsection{Non-unipotent representation}
In the following we assume that $\ell \neq 2$.
For $\chi \in \mu_{q+1}^{\vee}$, 
let $s_{\chi} \in \SO_{2n+1}(\bF_q)$ 
be a semisimple element 
corresponding to $\chi$ 
such that 
$\Spec(s_{\chi})=\{1,\ldots,1,\zeta_{\chi},\zeta_{\chi}^{-1} \}$ 
for $\zeta_{\chi} \in \mu_{q+1}$ 
and 
\begin{equation}\label{eq:Cents}
C_{\SO_{2n+1}(\bF_q)}(s_{\chi})= 
\begin{cases}
\SO_{2n-1}(\bF_q) \times \oU_1(\bF_q) & \textrm{if $\chi^2 \neq 1$,} \\ 
\SO_{2n-1}(\bF_q) \times \oO_2^-(\bF_q) & \textrm{if $p \neq 2$ and $\chi = \nu$.} 
\end{cases}
\end{equation}
We have 
\begin{equation}\label{eq:lsW}
W_{n,\psi}[\chi] \in \mathcal{E}(\Sp_{2n},(s_{\chi})) \quad 
\textrm{if $\chi^2 \neq 1$}, \quad 
W_{n,\psi}[\chi]^{\kappa} \in \mathcal{E}(\Sp_{2n},(s_{\chi})) \quad 
\textrm{if $\chi^2 =1$} 
\end{equation} 
by \cite[Proposition 7.12]{ITGeomHW}. 
We write as $q+1=\ell^a r$ with $(\ell,r)=1$. 

\begin{prop}\label{prop:nc1}
Let 
$\chi \in \mu_{q+1}^{\vee}$. 
We write as 
$\chi=\chi_{\ell^a} \chi_r$ as before. 
\begin{enumerate}
\item\label{en:Birchi} 
If $\chi_r^2 \neq 1$, the Brauer character 
$\overline{W_{n,\psi}[\chi]}$ 
is irreducible. 
\item 
Assume $p \neq 2$. 
For $\kappa \in \{\pm\}$, the Brauer character 
$\overline{W_{n,\psi}[\nu]^{\kappa}}$
is irreducible. 
\end{enumerate}
\end{prop}
\begin{proof}
To show the claim \ref{en:Birchi}, 
we may assume that $\chi_{\ell^a}=1$ 
by Corollary \ref{sde}. 
Then the claims follow from 
Lemma \ref{lem:lirr} using Lemma \ref{lem:dim}, 
\eqref{eq:Cents} and \eqref{eq:lsW}. 
\end{proof}

\begin{lem}\label{lem:dif}
The Brauer characters  
$\overline{W_{n,\psi}[\nu]^+}$ and 
$\overline{W_{n,\psi}[\nu]^-}$ are different. 
\end{lem}
\begin{proof}
It suffices to show that the characters $W_{n,\psi}[\nu]^+$ and $W_{n,\psi}[\nu]^-$ are distinct restricted to the subset consisting of $\ell$-regular elements of $\Sp_{2n}(\bF_q)$. 
The element $g_0$ in Proposition \ref{433} is of order $p$ and $\ell$-regular by $(p,\ell)=1$. 
Hence the claim follows from Proposition \ref{433}. 
\end{proof}

\subsection{Unipotent representation}

\begin{lem}\label{irr}
Assume that $n \geq 2$. 
Then the $\Sp_{2n}(\bF_q)$-representation  
$H_{\mathrm{c}}^{2n-1}(Y'_{2n,\overline{\mathbb{F}}_q},\ol{\bF}_{\ell})^-$ 
is irreducible. 
\end{lem}
\begin{proof}
The $\Sp_{2n}(\bF_q)$-representation $W_{n}[1]^{-}$ 
is irreducible modulo $\ell$ by Lemma \ref{lem:dim} and 
\cite[(6), Corollary 7.5]{GuTiCross} if $p=2$ and 
\cite[Corollary 7.4]{GMSTCross} if $p \neq 2$. 
Hence 
the claim follows from \eqref{irr-2}. 
\end{proof}
Assume that $\ell \mid q+1$ and $n \geq 2$. 
By Proposition \ref{lp} \ref{en:KFd} and Lemma \ref{ccc0} \ref{en:evtri}, 
we have a short exact sequence 
\begin{equation}\label{ses}
0 \to 
H_{\mathrm{c}}^{2n-1}(Y'_{2n,\overline{\mathbb{F}}_q},\mathscr{K}_{\chi_{\ell^a}}) \otimes_{\mathcal{O}} \ol{\bF}_{\ell} \to 
H_{\mathrm{c}}^{2n-1}(Y'_{2n,\overline{\mathbb{F}}_q},\ol{\bF}_{\ell})  \xrightarrow{\delta} \mathbf{1} \to 0  
\end{equation}
for any non-trivial character $\chi_{\ell^a}$. 
By Lemma \ref{irr}, the restriction of $\delta$
to $H_{\mathrm{c}}^{2n-1}(Y'_{2n,\overline{\mathbb{F}}_q},\ol{\bF}_{\ell})^-$ is a zero map. 
We denote by $\delta^+$ the restriction of $\delta$ to
$H_{\mathrm{c}}^{2n-1}(Y'_{2n,\overline{\mathbb{F}}_q},\ol{\bF}_{\ell})^+$. 
Then we have 
a short exact sequence 
\begin{equation}\label{ses2}
0 \to \ker \delta^+ \to 
H_{\mathrm{c}}^{2n-1}(Y'_{2n,\overline{\mathbb{F}}_q},\ol{\bF}_{\ell})^+ \to \mathbf{1} \to 0. 
\end{equation}

\begin{prop}\label{sum}
Assume that $p=2$ and $n \geq 2$. 
If $\ell \nmid q+1$, the $\Sp_{2n}(\bF_q)$-representation $H_{\mathrm{c}}^{2n-1}(Y'_{2n,\overline{\mathbb{F}}_q},\ol{\bF}_{\ell})^+$ is irreducible. 

Assume $\ell \mid q+1$. 
Then, the $\Sp_{2n}(\bF_q)$-representation  
$\ker \delta^+$ is irreducible. 
The representation 
$H_{\mathrm{c}}^{2n-1}(Y'_{2n,\overline{\mathbb{F}}_q},\ol{\bF}_{\ell})^+$ is indecomposable of length two
with irreducible constituents 
$\mathbf{1}$ and $\ker \delta^+$.
\end{prop}
\begin{proof}
We have 
\[
 W_{n,\psi}[1]^-=\alpha_n, \quad W_{n,\psi}[1]^+=\beta_n 
\]
in the notation of \cite[Definition (6)]{GuTiCross}
by Lemma \ref{lem:dim}, 
\cite[(4)]{GuTiCross} 
and \cite[Lemma 4.1]{TiZaSomeW}. 

Assume $\ell \nmid q+1$. By Proposition \ref{lp} \ref{en:KFnd}, 
we have 
\[
H_{\mathrm{c}}^{2n-1}(Y'_{2n,\overline{\mathbb{F}}_q},\mathcal{O}) \otimes_{\mathcal{O}} \ol{\bF}_{\ell} \simeq 
H_{\mathrm{c}}^{2n-1}(Y'_{2n,\overline{\mathbb{F}}_q},\ol{\bF}_{\ell}). 
\]
The representation 
$H_{\mathrm{c}}^{2n-1}(Y'_{2n,\overline{\mathbb{F}}_q},\ol{\bF}_{\ell})^+$ is irreducible by \eqref{irr-2} and 
\cite[Corollary 7.5 (i)]{GuTiCross}. 

Assume $\ell \mid q+1$. 
By  \eqref{irr-2} and \cite[Corollary 7.5 (i)]{GuTiCross}, 
$H_{\mathrm{c}}^{2n-1}(Y'_{2n,\overline{\mathbb{F}}_q},\ol{\bF}_{\ell})^+$ 
has two irreducible constituents. 
Hence, $\ker \delta^+$ is irreducible by \eqref{ses2}. 

We show that 
$H_{\mathrm{c}}^{2n-1}(Y'_{2n,\overline{\mathbb{F}}_q},\ol{\bF}_{\ell})^+$ 
is indecomposable. 
Assume that it is not so. 
Then it is completely reducible, and is isomorphic to a direct sum of 
$\ker \delta^+$ and $\mathbf{1}$ 
by the Jordan--H\"{o}lder theorem. 
This is contrary to  Proposition \ref{sol} \ref{en:invY}. 
Hence, we obtain the claim. 
\end{proof}

\begin{prop}\label{pl}
Assume that $n \geq 2$, $p \neq 2$ and $\ell \mid q+1$. 
The $\Sp_{2n}(\bF_q)$-representation $\ker \delta^+$
is irreducible. 
The $\Sp_{2n}(\bF_q)$-representation 
$H_{\mathrm{c}}^{2n-1}(Y'_{2n,\overline{\mathbb{F}}_q},\ol{\bF}_{\ell})^+$ is indecomposable of length two with irreducible constituents $\mathbf{1}$
and $\ker \delta^+$. 
\end{prop}
\begin{proof}
Let $U$ be as in \S \ref{ssec:Inv}. 
Since $U(\bF_q)^n$ is a $p$-group and $\ell \neq p$, 
any $U(\bF_q)^n$-representation over $\ol{\bF}_{\ell}$ is semisimple. 
By Proposition \ref{sol} \ref{en:invY} and \eqref{ses2}, 
we have 
\begin{equation}\label{plus}
\dim 
(H_{\mathrm{c}}^{2n-1}(Y'_{2n,\overline{\mathbb{F}}_q},\ol{\bF}_{\ell})^+)^{U(\bF_q)^n}=1. 
\end{equation}

We set $m=(q^n-q)(q^n-1)/(2(q+1))$. 
We assume that $\ker \delta^+$ has more than one irreducible constituents. 
By the assumption $p \neq 2$, $\ell \neq 2$ and 
$\ell \mid q+1$, we have $q>3$. 
We have $\dim \ker \delta^+=m+q^n-1<2m$.  
Hence we can take an irreducible constituent of 
$\ker \delta^+$ whose dimension is less than $m$, for which we write 
$\beta$. 
By \eqref{ses2} and \eqref{plus}, we have 
$\dim \beta^{U(\bF_q)^n}=0$. 
Hence, $\beta$ is not a trivial representation of $\Sp_{2n}(\bF_q)$. 
By \cite[Theorem 2.1]{GMSTCross}, $\beta$ must be a Weil module.  
The $\Sp_{2n}(\bF_q)$-representation 
$H_{\mathrm{c}}^{2n-1}(Y'_{2n,\overline{\mathbb{F}}_q},\overline{\mathbb{Q}}_{\ell})^+
$ is unipotent by \eqref{eq:X'Lpsi}, 
Lemma \ref{tl2} and \cite[Corollary 7.13]{ITGeomHW}. 
Hence $\ker \delta^+$ belongs 
to a unipotent block. 
Since $\beta$ and $\ker \delta^+$ belong to the same block,  
this is contrary to Lemma \ref{weil}. 
Hence, we obtain the first claim. 

By \eqref{ses2} and the first claim, 
$H_{\mathrm{c}}^{2n-1}(Y'_{2n,\overline{\mathbb{F}}_q},\ol{\bF}_{\ell})^+$
has length two. The sequence \eqref{ses2} is non-split 
by Proposition \ref{sol} \ref{en:invY}. 
Hence the second claim follows.  
\end{proof}

\begin{lem}\label{canon}
Let 
$\psi \in \Hom (\mathbb{F}_{q,+},\mathbb{F}^{\times}) \setminus \{1\}$. 
The canonical map 
\[
H_{\mathrm{c}}^n(\mathbb{A}^n,\pi^\ast \mathscr{L}_{\psi})
\to 
H^n(\mathbb{A}^n,\pi^\ast \mathscr{L}_{\psi})
\]
is an isomorphism. 
\end{lem}
\begin{proof}
Let $C$ be the affine curve over $\mathbb{F}_q$ defined by $z^q+z=t^{q+1}$ 
in $\mathbb{A}^2_{\mathbb{F}_q}$. 
We have $H_{\mathrm{c}}^1(C_{\overline{\mathbb{F}}_q},\mathbb{F})[\psi] 
\simeq H_{\mathrm{c}}^1(\mathbb{A}^1,\mathscr{L}_{\psi})$. 
Hence by the K\"{u}nneth formula, 
we have 
\[
H_{\mathrm{c}}^n(\mathbb{A}^n,\pi^\ast \mathscr{L}_{\psi}) \simeq (H_{\mathrm{c}}^1(C_{\overline{\mathbb{F}}_q},\mathbb{F})[\psi])^{\otimes n}, \quad 
H^n(\mathbb{A}^n,\pi^\ast \mathscr{L}_{\psi}) \simeq (H^1(C_{\overline{\mathbb{F}}_q},\mathbb{F})[\psi])^{\otimes n}. 
\]
Hence it suffices to show 
that the canonical map 
$H_{\mathrm{c}}^1(C_{\overline{\mathbb{F}}_q},\mathbb{F}) \to H^1(C_{\overline{\mathbb{F}}_q},\mathbb{F})$
is an isomorphism. 
The curve $C$ has the smooth compactification 
$\overline{C}$ defined by $X^qY+XY^q=Z^{q+1}$
in $\mathbb{P}_{\mathbb{F}_q}^2$. 
The complement 
$\overline{C} \setminus C$ consists of an 
$\mathbb{F}_q$-valued point.  
Hence, the claim follows.  
\end{proof}

\begin{lem}\label{sd}
Let 
$\psi \in \Hom (\mathbb{F}_q,\mathbb{F}^{\times}) \setminus \{1\}$. 
Then 
$H_{\mathrm{c}}^{2n}(\mathbb{A}^{2n},\pi'^\ast \mathscr{L}_{\psi})[1]$ 
is a self-dual representation of $\Sp_{2n}(\bF_q)$. 
\end{lem}
\begin{proof}
By Poincar\'{e} duality, 
we have an isomorphism 
\[
H^{2n}(\mathbb{A}^{2n},\pi'^\ast \mathscr{L}_{\psi}) \simeq 
H_{\mathrm{c}}^{2n}(\mathbb{A}^{2n},\pi'^\ast \mathscr{L}_{\psi^{-1}})^{\vee}. 
\] 
Hence the claim follows from 
\eqref{XLpsi}, \eqref{eq:X'Lpsi}, Lemma \ref{canon} and 
\cite[Remark 3.2]{ITGeomHW}. 
\end{proof}

\begin{prop}\label{canon2}
Assume that $n \geq 2$, $p \neq 2$ and $\ell \nmid q+1$. 
\begin{enumerate}
\item The $\Sp_{2n}(\bF_q)$-representation 
$H_{\mathrm{c}}^{2n-1}(Y'_{2n,\overline{\mathbb{F}}_q},\ol{\bF}_{\ell})^+$ is irreducible. 
\item For each $\kappa \in \{\pm\}$, 
the $\Sp_{2n}(\bF_q)$-representation 
$H_{\mathrm{c}}^{2n-1}(Y'_{2n,\overline{\mathbb{F}}_q},\ol{\bF}_{\ell})^{\kappa}$ is self-dual. 
\end{enumerate}
\end{prop}
\begin{proof}
For $\kappa \in \{\pm\}$, we simply write 
$W^{\kappa}$ for $H_{\mathrm{c}}^{2n-1}(Y'_{2n,\overline{\mathbb{F}}_q},\ol{\bF}_{\ell})^{\kappa}$.

We show the first claim. 
Assume  $(n,q)=(2,3)$. 
We have $\lvert \Sp_4(\bF_3) \rvert =2^7\cdot 3^4 \cdot 5$ and 
$\dim W^+=15$ by Lemma \ref{lem:dim}. 
By the assumption, we have $\ell \neq 2,3$. 
Hence, the claim in this case follows from 
the Brauer--Nesbitt theorem. 

Assume  $(n,q) \neq (2,3)$. 
Let $m$ be as in the proof of Proposition \ref{pl}. 
Assume that $W^+$ is not irreducible. 
By $(n,q) \neq (2,3)$, we can take an irreducible component 
$\beta$ of $W^+$ whose dimension is less than 
$m$. 
Then $\beta$ is a trivial module or a Weil module
by \cite[Theorem 2.1]{GMSTCross}. 
We know that $\beta$ is a trivial module 
by Lemma \ref{weil}. 
By the last isomorphism in 
Proposition \ref{sol} \ref{en:invY}, $W^+$ has at most one 
trivial module as irreducible constituents. 
Hence, $W^+$ must have length two by a similar argument as above.  
By $(W^+)^{\Sp_{2n}(\bF_q)}=0$ as in Proposition \ref{sol} \ref{en:invY}, 
we have a non-split surjective homomorphism 
$W^+ \twoheadrightarrow \mathbf{1}$. 
We set $W=W^+\oplus W^-$. 
Since $W$ is self-dual by Lemma \ref{tl2} and Lemma \ref{sd}, 
we have an injective homomorphism 
$\mathbf{1} \hookrightarrow W$. 
By taking the $\Sp_{2n}(\bF_q)$-fixed part of this, 
we have $W^{\Sp_{2n}(\bF_q)} \neq 0$, which is contrary 
to Proposition \ref{sol} \ref{en:invY}. 
Hence $W^+$ is irreducible. 

We show the second claim. 
For each $\kappa \in \{\pm\}$, the $\Sp_{2n}(\bF_q)$-representation $W^\kappa$ 
is irreducible by Lemma \ref{irr} and the first claim. 
Since the dimensions of $W^+$ and $W^-$ are different, 
we obtain the claim by the self-duality of 
$W$. 
\end{proof}

\section{Mod $\ell$ Howe correspondence}\label{mod3} 
We formulate a mod $\ell$ Howe correspondence 
for 
$(\Sp_{2n},\oO_2^-)$ 
using mod $\ell$ cohomology of $Y'_{2n}$, 
and show that it is compatible with the ordinary Howe correspondence. 
 
\subsection{Representation of $\oO_2^-(\bF_q)$}\label{mmod}
Let $W=\mathbb{F}_{q^2}$. 
We consider the quadratic form 
$Q \colon W \to \mathbb{F}_q;\ x \mapsto x^{q+1}$. 
Recall that $\oO_2^-$ is 
the orthogonal group over $\bF_q$ defined by $Q$. 
Clearly, we have $Q(\zeta x)=Q(x)$ for any $x \in W$
and $\zeta \in \mu_{q+1}$. 
Hence, we have a natural inclusion 
$\mu_{q+1} \hookrightarrow \oO_2^-(\bF_q)$. 
We regard 
$F_W \colon W \to W;\ x \mapsto x^q$ as an element of $\oO_2^-(\bF_q)$. 
We can easily check that 
$\mu_{q+1} \cap \langle F_W\rangle=\{1\}$. 
This group $\oO_2^-(\bF_q)$ 
is isomorphic to the dihedral group 
of order $2(q+1)$ 
by \cite[Proposition 2.9.1]{KlLiSubcl}. 
We fix the isomorphism 
\[
 \mu_{q+1} \rtimes (\bZ/2\bZ) \xrightarrow{\sim} \oO_2^-(\bF_q);\ 
 (\zeta,i) \mapsto \zeta F_W^i. 
\]
For a pair 
$(\xi,\kappa) \in \Hom (\mu_{q+1},\mu_2(\ol{\bF}_{\ell})) \times \{\pm\}$ such that 
$\xi^2=1$, the map 
\[
(\xi,\kappa) \colon \oO_2^-(\bF_q) \to \mu_2(\ol{\bF}_{\ell});\ 
(x,k) \mapsto \kappa^k \xi (x)
\] 
for $x \in \mu_{q+1}$ 
and $k \in \mathbb{Z}/2\mathbb{Z}$ 
is a character. 
For a character $\xi \in 
\Hom (\mu_{q+1},\ol{\bF}_{\ell}^{\times})$ such that 
$\xi^2 \neq 1$, the $2$-dimensional 
representation 
$\sigma_{\xi}=\Ind_{\mu_{q+1}}^{\oO_2^-(\bF_q)} \xi$ 
is irreducible. 
Note that 
$\sigma_{\xi} \simeq \sigma_{\xi^{-1}}$ as 
$\oO_2^-(\bF_q)$-representations. 
Any irreducible representation of $\oO_2^-(\bF_q)$ 
is isomorphic to one of these representations. 

\subsection{Formulation}\label{sec:Formulation}
Let $\Irr_{\ol{\bF}_{\ell}}(\oO_2^-(\bF_q))$ 
be the set 
of irreducible representations of 
$\oO_2^-(\bF_q)$ over $\ol{\bF}_{\ell}$. 
Let $1 \in \bZ/2\bZ$ act on $\Hom (\mu_{q+1},\ol{\bF}_{\ell}^{\times})$ 
by $\xi \mapsto \xi^{-1}$. 
Then $\Irr_{\ol{\bF}_{\ell}}(\oO_2^-(\bF_q))$ 
is parametrized by 
\[
 \{ \xi \in \Hom (\mu_{q+1},\ol{\bF}_{\ell}^{\times}) \mid 
 \xi^2 \neq 1 \}/(\mathbb{Z}/2\mathbb{Z}) \cup 
 \{ (\xi,\kappa) \mid 
 \xi \in \Hom (\mu_{q+1},\mu_2 (\ol{\bF}_{\ell})), 
 \kappa \in \{ \pm \} \} 
\]
as in \S \ref{mmod}. 

Assume  $n \geq 2$. We define a mod $\ell$ Howe correspondence 
\[
 \Theta_{\ell} \colon \Irr_{\ol{\bF}_{\ell}}(\oO_2^-(\bF_q)) \to 
 \{ \textrm{the representations of 
 $\Sp_{2n}(\bF_q)$ over $\ol{\bF}_{\ell}$} \}
\]
by 
\[
 [\xi] \mapsto 
 H_{\mathrm{c}}^{2n-1}(Y'_{2n,\overline{\mathbb{F}}_q},\mathscr{K}_{\xi}) , \quad 
 (\xi,\kappa) \mapsto 
 H_{\mathrm{c}}^{2n-1}(Y'_{2n,\overline{\mathbb{F}}_q},\mathscr{K}_{\xi})^{\kappa} . 
\]

\begin{thm}\label{thm}
Let $\tau$ be an irreducible representation of 
$\oO_2^-(\bF_q)$ over $\ol{\bF}_{\ell}$. 
Then 
$\Theta_{\ell}(\tau)$ is irreducible 
except the case where $\ell \mid q+1$ and 
$\tau$ corresponds to $(1,+)$, 
in which case $\Theta_{\ell}(\tau)$ is 
a non-trivial extension of 
the trivial representation 
by an irreducible representation. 
Furthermore, 
if $\tau, \tau' \in \Irr_{\ol{\bF}_{\ell}}(\oO_2^-(\bF_q))$ 
are different, 
$\Theta_{\ell}(\tau)$ and $\Theta_{\ell}(\tau')$ 
have no irreducible constituent in common. 
\end{thm}
\begin{proof}
The first claim follows from 
Proposition \ref{prop:nc1}, 
Lemma \ref{irr}, 
Proposition \ref{sum}, 
Proposition \ref{pl} and 
Proposition \ref{canon2}. 

By Lemma \ref{lem:diflb} and 
\eqref{eq:lsW} for $\chi \in \mu_r^{\vee}$, 
the representations 
$H_{\mathrm{c}}^{2n-1}(Y'_{2n,\overline{\mathbb{F}}_q},\mathscr{K}_{\xi})$ 
of $\Sp_{2n}(\bF_q)$ 
for $\xi \in \Hom (\mu_r,\ol{\bF}_{\ell}^{\times})$ 
have no irreducible constituent in common. 
Therefore the second claim follows from 
Lemma \ref{lem:dim}, \eqref{irr-2} and 
Lemma \ref{lem:dif}. 
\end{proof}

We extend $\Theta_{\ell}$ 
to the set of finite-dimensional 
semisimple representations of 
$\oO_2^-(\bF_q)$ over $\ol{\bF}_{\ell}$ 
by additivity. 
Let $\Theta$ be the Howe correspondence for 
$\Sp_{2n}(\bF_q) \times \oO_2^-(\bF_q)$ (\cf \cite[\S 7.2]{ITGeomHW}). 

\begin{prop}\label{lasth}
Let $\pi$ be an irreducible representation of 
$\oO_2^-(\bF_q)$ over $\ol{\bQ}_{\ell}$. 
We have an injection 
\[
 \overline{\Theta (\pi)}^{\mathrm{ss}} 
 \hookrightarrow 
 \Theta_{\ell} (\overline{\pi}^{\mathrm{ss}}) , 
\]
where $\overline{(-)}^{\mathrm{ss}}$ 
denotes the semi-simplification of a mod 
$\ell$ reduction. 
The injection is an isomorphism 
except the cases where 
$\pi$ corresponds to $\chi=\chi_r \chi_{\ell^a}$ 
and we have $\chi_r=1$, $\chi_{\ell^a}\neq 1$. 
\end{prop}
\begin{proof}
This follows from \eqref{irr-2}. 
\end{proof}

\begin{rem}
A mod $\ell$ Howe correspondence is studied in \cite{AubSerHC} in 
a different way and in a general setting under $p \neq 2$ up to semi-simplifications. 
\end{rem}

\noindent
Naoki Imai\\
Graduate School of Mathematical Sciences, The University of Tokyo, 
3-8-1 Komaba, Meguro-ku, Tokyo, 153-8914, Japan \\
naoki@ms.u-tokyo.ac.jp \\[0.5cm]
Takahiro Tsushima\\ 
Keio University School of Medicine, 4-1-1 Hiyoshi, Kohoku-ku, Yokohama, 223-8521, Japan \\
tsushima@keio.jp


\begin{thebibliography}{GMST02}
	\providecommand{\url}[1]{\texttt{#1}}
	\providecommand{\urlprefix}{URL }
	\providecommand{\eprint}[2][]{\url{#2}}
	
	\bibitem[Aub94]{AubSerHC}
	A.-M. Aubert, S\'{e}ries de {H}arish-{C}handra de modules et correspondance de
	{H}owe modulaire, J. Algebra 165 (1994), no.~3, 576--601.
	
	\bibitem[BM89]{BrMiBlL}
	M.~Brou\'{e} and J.~Michel, Blocs et s\'{e}ries de {L}usztig dans un groupe
	r\'{e}ductif fini, J. Reine Angew. Math. 395 (1989), 56--67.
	
	\bibitem[Bon11]{BonRepSL2}
	C.~Bonnaf\'{e}, Representations of {${\rm SL}_2(\Bbb F_q)$}, vol.~13 of Algebra
	and Applications, Springer-Verlag London, Ltd., London, 2011.
	
	\bibitem[Dim92]{DimSinghyp}
	A.~Dimca, Singularities and topology of hypersurfaces, Universitext,
	Springer-Verlag, New York, 1992.
	
	\bibitem[DM91]{DiMiRepLie}
	F.~Digne and J.~Michel, Representations of finite groups of {L}ie type, vol.~21
	of London Mathematical Society Student Texts, Cambridge University Press,
	Cambridge, 1991.
	
	\bibitem[FS82]{FoSrBlgu}
	P.~Fong and B.~Srinivasan, The blocks of finite general linear and unitary
	groups, Invent. Math. 69 (1982), no.~1, 109--153.
	
	\bibitem[G{\'{e}}r77]{GerWeil}
	P.~G{\'{e}}rardin, Weil representations associated to finite fields, J. Algebra
	46 (1977), no.~1, 54--101.
	
	\bibitem[GMST02]{GMSTCross}
	R.~M. Guralnick, K.~Magaard, J.~Saxl and P.~H. Tiep, Cross characteristic
	representations of symplectic and unitary groups, J. Algebra 257 (2002),
	no.~2, 291--347.
	
	\bibitem[GT04]{GuTiCross}
	R.~M. Guralnick and P.~H. Tiep, Cross characteristic representations of even
	characteristic symplectic groups, Trans. Amer. Math. Soc. 356 (2004), no.~12,
	4969--5023.
	
	\bibitem[HM78]{HoMaTT}
	R.~Hotta and K.~Matsui, On a lemma of {T}ate-{T}hompson, Hiroshima Math. J. 8
	(1978), no.~2, 255--268.
	
	\bibitem[HM01]{HiMaLowuni}
	G.~Hiss and G.~Malle, Low-dimensional representations of special unitary
	groups, J. Algebra 236 (2001), no.~2, 745--767.
	
	\bibitem[How73]{HoweCharW}
	R.~E. Howe, On the character of {W}eil's representation, Trans. Amer. Math.
	Soc. 177 (1973), 287--298.
	
	\bibitem[Ill81]{IllBrEP}
	L.~Illusie, Th\'{e}orie de {B}rauer et caract\'{e}ristique
	d'{E}uler-{P}oincar\'{e} (d'apr\`es {P}. {D}eligne), in The
	{E}uler-{P}oincar\'{e} characteristic ({F}rench), vol.~82 of Ast\'{e}risque,
	pp. 161--172, Soc. Math. France, Paris, 1981.
	
	\bibitem[IT23]{ITGeomHW}
	N.~Imai and T.~Tsushima, Geometric construction of {H}eisenberg-{W}eil
	representations for finite unitary groups and {H}owe correspondences, Eur. J.
	Math. 9 (2023), no.~2, Paper No. 31, 34.
	
	\bibitem[IT24]{ITShinlift}
	N.~Imai and T.~Tsushima, Shintani lifts for Weil representations of unitary
	groups over finite fields, Math. Res. Lett. 31 (2024), no.~5, 1471--1491.
	
	\bibitem[KL90]{KlLiSubcl}
	P.~Kleidman and M.~Liebeck, The subgroup structure of the finite classical
	groups, vol. 129 of London Mathematical Society Lecture Note Series,
	Cambridge University Press, Cambridge, 1990.
	
	\bibitem[LOST10]{LOSTOre}
	M.~W. Liebeck, E.~A. O'Brien, A.~Shalev and P.~H. Tiep, The {O}re conjecture,
	J. Eur. Math. Soc. (JEMS) 12 (2010), no.~4, 939--1008.
	
	\bibitem[Lus77]{LusIrrcl}
	G.~Lusztig, Irreducible representations of finite classical groups, Invent.
	Math. 43 (1977), no.~2, 125--175.
	
	\bibitem[Sai72]{SaiRepsym}
	M.~Saito, Repr\'{e}sentations unitaires des groupes symplectiques, J. Math.
	Soc. Japan 24 (1972), 232--251.
	
	\bibitem[Shi80]{ShiChWe}
	K.~Shinoda, The characters of {W}eil representations associated to finite
	fields, J. Algebra 66 (1980), no.~1, 251--280.
	
	\bibitem[TZ97]{TiZaSomeW}
	P.~H. Tiep and A.~E. Zalesskii, Some characterizations of the {W}eil
	representations of the symplectic and unitary groups, J. Algebra 192 (1997),
	no.~1, 130--165.
	
	\bibitem[Wei64]{Weiopuni}
	A.~Weil, Sur certains groupes d'op\'{e}rateurs unitaires, Acta Math. 111
	(1964), 143--211.
	
	\bibitem[Whi90]{WhiDecndiv}
	D.~L. White, Decomposition numbers of {${\rm Sp}(4,q)$} for primes dividing
	{$q\pm 1$}, J. Algebra 132 (1990), no.~2, 488--500.
	
\end{thebibliography}
\end{document}